\documentclass{amsart}

\usepackage{stmaryrd}

\usepackage{amsmath, amsthm, amssymb}

\theoremstyle{definition}%non-italic in Theorem%

\usepackage{scalerel}%Making smaller subscripts%
\usepackage{mathrsfs}

\usepackage{tikz}
\usepackage{hyperref}%hyperlink%
\usepackage{placeins}%FloatBarrier%

\usepackage{proof}

\usepackage{verbatim}%Comment%
\usepackage{mathtools}
\usepackage{cleveref}
\usepackage{xspace}

\usepackage{bbm}

\usepackage[bottom]{footmisc}

\usepackage{enumitem}

\usepackage{biblatex} %bib%
\newtheorem{theorem}{Theorem}
\newtheorem*{theorem*}{Theorem}%unnumered theorem%
\newtheorem*{lemma*}{Lemma}%unnumered lemma%
\newtheorem*{prop*}{Proposition}%unnumered proposition%
\newcommand{\namedresultname}{Result}
\newtheorem{namedresultinner}[theorem]{\namedresultname}
\newtheorem*{namedresultinnerstar}{\namedresultname}

\newenvironment{namedresult*}[1]{%
  \renewcommand{\namedresultname}{#1}%
  \begin{namedresultinnerstar}%
}{\end{namedresultinnerstar}}
\newtheorem*{maintheorem*}{Main Theorem}%unnumered maintheorem%

\newtheorem*{maintheorems*}{Main Theorems}

\newtheorem{definition}[theorem]{Definition}
\newtheorem{corollary}{Corollary}[theorem]
\newtheorem{prop}[theorem]{Proposition}
\newtheorem{lemma}[theorem]{Lemma}

\newtheorem{remark}[theorem]{Remark}
\newtheorem{claim}{Claim}[theorem]
\newcounter{implicationdiagram}

\crefname{theorem}{Theorem}{Theorems}
\crefname{lemma}{Lemma}{Lemmas}
\crefname{prop}{Proposition}{Propositions}
\crefname{definition}{Definition}{Definitions}
\crefname{corollary}{Corollary}{Corollaries}
\crefname{claim}{Claim}{Claims}
\crefname{remark}{Remark}{Remarks}
\crefname{section}{Section}{Sections}
\crefname{subsection}{Subsection}{Subsections}
\crefname{implicationdiagram}{Implication Diagram}{Implication Diagrams}

\newcommand\Stage{\textup{Stage}}
\newcommand\form{\preccurlyeq}
\newcommand{\avail}{\mathrel{\triangleleft}}
\newcommand\Slice{\textup{Slice}}
\newcommand\Lev{\textup{Lev}}
\newcommand\SemiLev{\textup{SemiLev}}
\newcommand\Pow{\mathcal{P}}

\newcommand\Ur{\textup{Ur}}
\newcommand\Set{\textup{Set}}
\newcommand\pot{\mathbin{\mathparagraph}}

\newcommand\Ord{\textup{Ord}}

\renewcommand{\ker}{\textup{ker}}

\newcommand\TheoryName[1]{\ensuremath{#1}\xspace}
\newcommand\TheoryHyphen{\text{-}\allowbreak}
\newcommand\TheorySpace{\allowbreak\ }
\xspaceaddexceptions{+}

\newcommand\RP{\TheoryName{\mathsf{RP}}}

\providecommand\ST{}
\renewcommand\ST{\TheoryName{\mathsf{ST}}}
\providecommand\LT{}
\renewcommand\LT{\TheoryName{\mathsf{LT}}}
\newcommand\STU{\TheoryName{\mathsf{STU}}}
\makeatletter
\DeclareRobustCommand\LTU{\@ifnextchar2{\LTU@two}{\TheoryName{\mathsf{LTU}}}}
\def\LTU@two2{\TheoryName{\mathsf{LTU}_{2}}}
\makeatother
\newcommand\LTUMinus{\TheoryName{\mathsf{LTU}^{-}}}
\newcommand\LTUPlus{\TheoryName{\mathsf{LTU}^{+}}}
\newcommand\ZF{\TheoryName{\mathsf{ZF}}}

\newcommand\AC{\TheoryName{\mathsf{AC}}}
\newcommand\UrDef{\TheoryName{\mathsf{UrDef}}}
\newcommand\UrStage{\TheoryName{\mathsf{Ur}\TheoryHyphen\mathsf{Stage}}}
\newcommand\UrSet{\TheoryName{\mathsf{Ur}\TheoryHyphen\mathsf{Set}}}
\newcommand\SpecU{\TheoryName{\mathsf{Spec}\TheoryHyphen\mathsf{U}}}
\newcommand\Extensionality{\TheoryName{\mathsf{Extensionality}}}
\newcommand\OrderAxiom{\TheoryName{\mathsf{Order}}}
\newcommand\Staging{\TheoryName{\mathsf{Staging}}}
\newcommand\Priority{\TheoryName{\mathsf{Priority}}}
\newcommand\Specification{\TheoryName{\mathsf{Specification}}}
\newcommand\Separation{\TheoryName{\mathsf{Separation}}}
\newcommand\SeparationTwo{\TheoryName{\mathsf{Separation}_{2}}}
\newcommand\Stratification{\TheoryName{\mathsf{Stratification}}}
\newcommand\StratificationMinus{\TheoryName{\mathsf{Stratification}^{-}}}
\makeatletter
\newcommand\Replacement{\@ifnextchar2{\Replacement@two}{\TheoryName{\mathsf{Replacement}}}}
\def\Replacement@two2{\TheoryName{\mathsf{Replacement}_{2}}}
\makeatother
\newcommand\Collection{\TheoryName{\mathsf{Collection}}}
\newcommand\CollectionTwo{\TheoryName{\mathsf{Collection}_{2}}}
\newcommand\Foundation{\TheoryName{\mathsf{Foundation}}}
\newcommand\Pairing{\TheoryName{\mathsf{Pairing}}}
\newcommand\Powerset{\TheoryName{\mathsf{Powerset}}}
\newcommand\UnionAxiom{\TheoryName{\mathsf{Union}}}
\newcommand\InfinityAxiom{\TheoryName{\mathsf{Infinity}}}

\newcommand\Tail{\TheoryName{\mathsf{Tail}}}

\newcommand\UnboundedPlus{\TheoryName{\mathsf{Unbounded}^{+}}}
\newcommand\StageDirected{\TheoryName{\mathsf{Stage}\TheoryHyphen\mathsf{Directed}}}
\newcommand\StageUnbounded{\TheoryName{\mathsf{Stage}\TheoryHyphen\mathsf{Unbounded}}}
\newcommand\StageUnboundedPlus{\TheoryName{\mathsf{Stage}\TheoryHyphen\mathsf{Unbounded}^{+}}}
\newcommand\StageRP{\TheoryName{\mathsf{Stage}\TheoryHyphen\mathsf{RP}}}
\newcommand\StageRPMinus{\TheoryName{\mathsf{Stage}\TheoryHyphen\mathsf{RP}^{\sim}}}
\newcommand\LevelDirected{\TheoryName{\mathsf{Level}\TheoryHyphen\mathsf{Directed}}}
\makeatletter
\newcommand\LevelUnbounded{\@ifnextchar2{\LevelUnbounded@two}{\TheoryName{\mathsf{Level}\TheoryHyphen\mathsf{Unbounded}}}}
\def\LevelUnbounded@two2{\TheoryName{\mathsf{Level}\TheoryHyphen\mathsf{Unbounded}_{2}}}
\makeatother
\newcommand\LevelUnboundedPlus{\TheoryName{\mathsf{Level}\TheoryHyphen\mathsf{Unbounded}^{+}}}
\newcommand\LevelRP{\TheoryName{\mathsf{Level}\TheoryHyphen\mathsf{RP}}}
\newcommand\LevelRPMinus{\TheoryName{\mathsf{Level}\TheoryHyphen\mathsf{RP}^{-}}}
\newcommand\LevelExtension{\TheoryName{\mathsf{Level}\TheorySpace\mathsf{Extension}}}
\newcommand\RPMinus{\TheoryName{\mathsf{RP}^{-}}}
\newcommand\RPSim{\TheoryName{\mathsf{RP}^{\sim}}}
\newcommand\RPSimCl{\TheoryName{\mathsf{RP}^{\sim}_{\mathsf{cl}}}}
\newcommand\RPCl{\TheoryName{\mathsf{RP}_{\mathsf{cl}}}}
\newcommand\PartialReflection{\TheoryName{\mathsf{Partial}\TheorySpace\mathsf{Reflection}}}
\newcommand\CompleteReflection{\TheoryName{\mathsf{Complete}\TheorySpace\mathsf{Reflection}}}
\newcommand\LevyMontagueReflection{\TheoryName{\mathsf{L\acute{e}vy}\TheoryHyphen\mathsf{Montague}\TheorySpace\mathsf{Reflection}}}

\makeatletter
\DeclareRobustCommand\ZFU{\@ifnextchar2{\ZFU@two}{\TheoryName{\mathsf{ZFU}}}}
\def\ZFU@two2{\TheoryName{\mathsf{ZFU}_{2}}}
\makeatother
\makeatletter
\newcommand\ZFCU{\@ifnextchar2{\ZFCU@two}{\TheoryName{\mathsf{ZFCU}}}}
\def\ZFCU@two2{\TheoryName{\mathsf{ZFCU}_{2}}}
\makeatother
\newcommand\ZFUR{\TheoryName{\mathsf{ZFU}_{\mathsf{R}}}}

\renewcommand{\restriction}{\mathord{\upharpoonright}}

\title{\textbf{The Iterative Conception Reconsidered}}
\author{Bokai Yao}\address[Bokai Yao]{Peking University}
 \email{bkyao@pku.edu.cn}
\urladdr{https://bokaiyao.com}
\thanks{The author used OpenAI Codex as an interactive research and writing aid to explore proof strategies and assist with LaTex editing. All mathematical arguments were independently verified and revised by the author, who takes full responsibility for the work. The  author was supported by NSFC No. 12401001 and the Fundamental Research Funds for the Central Universities, Peking University.}

\begin{document}

\begin{abstract}
We investigate the iterative conception of set in its most general form, allowing urelements without assuming that they form a set. We formulate this conception in two ways, as stage theory and as level theory, and develop a general theory of levels with urelements. Unlike their pure-set counterparts, the resulting stage and level theories are not set-theoretically equivalent; moreover, second-order level theory with urelements is not weakly quasi-categorical. We then consider further principles governing stages and levels, motivated by directedness, unboundedness, and reflection. Some of these principles restore set-theoretic equivalence between the corresponding theories, while their level-theoretic versions yield forms of quasi-categoricity. These principles form strict implication hierarchies, thereby revealing distinct stronger conceptions of set beyond the basic iterative conception
\end{abstract}

\maketitle

\tableofcontents

\section{A Myth of the Iterative Conception}\label{Section:Intro}
According to the iterative conception of set, every set is formed at a stage from objects available at that stage. Many standard axioms of set theory are thought to be justified by some version of this conception. But what is the iterative conception of set, exactly? The following principle is commonly assumed to be part of the iterative conception.
\begin{itemize}
 \item [] \UrStage There is a stage at which all urelements are available.
\end{itemize}
\textit{Urelements} are members of sets that are not themselves sets. The picture seems natural: before any sets exist, there is a first stage at which all the urelements are available. The story goes: since the objects available at any stage form a set at that stage, the urelements will form a set. Thus \UrStage justifies the following axiom.
\begin{itemize}
 \item [] \UrSet There is a set of all urelements.
\end{itemize}
Most set theorists focus on the universe of \textit{pure sets} as generated by the iterative conception. Here the iterative process begins with an initial stage, below every other stage, at which nothing is available; only the empty set is formed at this stage, and no urelements ever enter the process. In this setting, both \UrStage and \UrSet hold trivially. Yet many philosophers interested in universes containing impure sets also accept \UrStage, often on the basis of a standard argument. Lewis \cite[104]{Lewis1986PluralityWorlds}, when considering whether there is a set of all possible worlds, gives one version of this argument:
\begin{quote}
But all the [urelements], no matter how many there may be, get in already on the ground floor. So, after all, we have no notion what could stop any class of [urelements] - in particular, the class of all worlds - from comprising a set. ... So I continue to accept a set of all worlds, indeed a set of all [urelements].
\end{quote}

\noindent Menzel \cite[64]{Menzel2014WideSets} gives a more direct version of this argument.
\begin{quote}
The argument for [\UrSet], then, is simply this: One of the sets that can be formed from the [urelements] is the set of all of them. They are, after all, all there to be collected at the very beginning of the process. Hence, the set of all [urelements] is formed at the very first stage.
\end{quote}

\noindent Consequently, \UrStage and \UrSet are taken as part of the iterative conception throughout most of the existing literature (e.g., Boolos \cite[221]{Boolos1971IterativeConception}, McGee \cite[52--55]{McGee1997MathematicalLanguage}, Potter \cite[40--41]{Potter2004SetTheoryPhilosophy}, and Incurvati \cite[40--41]{Incurvati2020Conceptions}).

Something is off. The iterative conception is about how sets come into existence. It is \textit{not} supposed to tell us how many urelements there are or whether they form a set, as these matters should be decided by theories of urelements. Indeed, when supplemented with \UrSet, the iterative conception comes into tension with certain metaphysical theories according to which the urelements do not form a set.\footnote{For example, if possible worlds (or possible individuals) are urelements, familiar recombination arguments suggest that there are too many of them to form a set. This worry has been pressed against Lewis's own modal realism \cite[104]{Lewis1986PluralityWorlds}; Nolan \cite{Nolan1996RecombinationUnbound} and Sider \cite[250]{Sider2009WilliamsonManyNecessaryExistents} develop versions of the point, and related arguments are central to Hawthorne and Uzquiano \cite{HawthorneUzquiano2011Angels} and Uzquiano \cite{Uzquiano2015RecombinationParadox}. Or suppose that the urelements include propositions or truths. Then the Russell--Myhill tradition and Grim's argument that there is no set of all truths give us another natural route to the conclusion that such urelements need not form a set (see Grim \cite{Grim1984NoSetTruths} and Kment \cite[58--59]{Kment2022RussellMyhillGrounding}).}  Note that the concern is not whether these metaphysical arguments are sound; it is that a basic conception of set should remain metaphysically neutral. 

It is a myth that \UrStage is part of the iterative conception. The standard argument for \UrStage, as quoted above, is flawed by making two unjustified assumptions: (i) there is a \textit{first stage}, a stage below every other stage, and (ii) all urelements are available on the first stage. Neither assumption, however, is warranted by the basic iterative conception. It is a surprising result, essentially due to Scott \cite[211--212]{Scott1974AxiomatizingSetTheory}, that the bare description of the iterative conception implies that the stages are \textit{well-founded}. Yet the iterative conception gives us no reason to suppose that this relation is a \textit{well-order}. A perfectly coherent version of the iterative story might contain pairwise \textit{incomparable} \textit{initial stages}, with no stage below any of them; and beginning from each initial stage, sets are formed from the urelements available on that stage.

Even if there is a first stage, it remains unjustified to assume that all urelements must be available on that stage. Consider \textit{mereological fusions} of sets and urelements. The mereological fusion of Socrates and $\{\text{Socrates}\}$, written as $\text{Socrates}+\{\text{Socrates}\}$, is presumably not a set but can nevertheless have a singleton. On this assumption, the fusion is a urelement. But it then seems bizarre for $\text{Socrates}+\{\text{Socrates}\}$ to be available at the first stage, since its part $\{\text{Socrates}\}$ is not yet available there! Thus the following more natural picture suggests itself. 

\begin{center}
\begin{tikzpicture}[
  scale=.8,
  transform shape,
  every node/.style={font=\small, align=center, inner sep=1pt}
]
\node (first) at (0,0) {$\text{Socrates},\ \ldots$};
\node (second) at (0,1.15) {$\{\text{Socrates}\},\
  \ \text{Socrates}+\{\text{Socrates}\},\ \ldots$};
\node (third) at (0,2.3) {$\{\text{Socrates}+\{\text{Socrates}\}\},\
  \ \text{Socrates}+\{\{\text{Socrates}\}\},\ \ldots$};
\node (above) at (0,3.25) {$\vdots$};
\node[anchor=east] at (-4.4,0) {$\mathrm{stage}_0$};
\node[anchor=east] at (-4.4,1.15) {$\mathrm{stage}_1$};
\node[anchor=east] at (-4.4,2.3) {$\mathrm{stage}_2$};
\end{tikzpicture}
\end{center}

\noindent That is, the urelement $\text{Socrates}+\{\text{Socrates}\}$ is available at the second stage and, together with other available objects, generates new sets and fusions. At each stage, new fusions are available as urelements, so both \UrStage and \UrSet fail in this picture. Again, the point is not to endorse this metaphysical picture, which some might reject (e.g., Lewis \cite{Lewis1991PartsClasses}), or \textit{argue against} \UrStage. Rather, this alternative picture shows that \UrStage is not metaphysically neutral and therefore should not \textit{be part of} a basic conception of set.

This paper examines the iterative conception of set with urelements in its most basic and general form without assuming \UrSet. In Section \ref{sec:levels and stages}, drawing on the presentations of Boolos \cite{Boolos1989IterationAgain} and Button \cite{Button2021LevelTheory1}, we formulate the basic iterative conception with urelements in terms of both stages and levels. We develop a more general theory of levels with urelements \LTU and establish some basic facts. As a subtheory of \ZFU, \LTU has various kinds of models in which \UrSet fails. In Section \ref{sec: Comparison}, we compare \LTU with its stage counterpart \STU. Button's equivalence result in the pure setting no longer holds: the urelement level theory \LTU proves more about sets than its stage counterpart \STU. We argue that this offers a first indication that the basic iterative conception is not robust. Section \ref{sec:extendingBIS} considers several standard ways of extending the basic iterative conception through principles of directedness, unboundedness, and reflection for stage and level theory. Section \ref{sec:Equivalence} shows that some natural extensions restore set-theoretic equivalence between the stage and level approaches. Using this result, we prove that these extensions form two strict implication hierarchies, revealing different conceptions of set based on the iterative conception. Section \ref{section:quasi-categoricity} turns to the quasi-categoricity of level theory with urelements. We distinguish two notions of quasi-categoricity. While the level theory for pure sets is quasi-categorical in the strong sense, as shown by Button and Walsh \cite{ButtonWalsh2018PhilosophyModelTheory}, second-order \LTU is not even weakly quasi-categorical. However, the principles in Section \ref{sec:extendingBIS} restore weak quasi-categoricity to level theory. The usual stronger notion of quasi-categoricity, however, depends on whether the Axiom of Choice holds the metatheory.

\section{Two Formalizations: Levels and Stages}\label{sec:levels and stages}
To formulate the iterative conception in its most basic and general form, we should treat it as a story about nothing more than how sets are formed. Following Button \cite{Button2021LevelTheory1}, let us call this \textit{the Basic Iterative Story}.\\

\begin{itemize}
\item [] \textbf{Basic Iterative Story} Every set is formed at some stage. At any stage and for any \textit{things} available at that stage, we form a set whose members are exactly those things. We form nothing else at a stage.\\
\end{itemize}

\noindent The story is absolutely general: the plural ``\textit{things}'' includes both sets and urelements, so the story allows non-pure sets, such as $\{\text{Socrates}\}$ and $\{\text{Socrates}, \{\text{Socrates}\}\}$, to form; the story does not posit a stage where all urelements are available.

We now consider two prominent ways of formalizing the Basic Iterative Story based on the two existing approaches. This leads to two basic formal theories, but they do not tell the same iterative story as we expect them to. We will argue that this divergence result shows that the basic iterative conception of set, in its most general form, fails to be robust.

\subsection{\STU}
\noindent  The most direct way to formalize the iterative story is to talk about stages explicitly. This is the route taken by Boolos in both of his classic presentations of the iterative conception \cite{Boolos1971IterativeConception} and \cite{Boolos1989IterationAgain}. 
\begin{definition}
\textit{The language of urelement set theory}, $\mathcal{L}_\Ur$, is the first-order language whose non-logical symbols are a binary predicate $\in$ and a unary predicate $\Ur$ for urelementhood.
\end{definition}
\noindent We read $\Ur(x)$ as ``$x$ is a urelement'', and write $\Set(x)$ for $\neg\Ur(x)$, meaning ``$x$ is a set''.
\begin{definition}
\textit{The language of stage theory},  $\mathcal{L}_\textbf{stage}$, is a two-sorted first-order language which extends $\mathcal{L}_\Ur$ by adding stage variables, a relation $<$ between stages, read as ``is below'', and a relation $\form$ between objects and stages.
\end{definition}
\noindent Following the convention, we use boldface variables $\mathbf{r},\mathbf{s},\mathbf{t}$, ... for stages.However, $x\form\mathbf{s}$ will be read as ``$x$ occurs at $\mathbf{s}$''. We adopt this different reading to handle urelements and distinguish \textit{being formed} from \textit{being available}. Conceptually,  a set $x$ occurs at a stage \textbf{s} just in case $x$ is formed at \textbf{s}. If \textbf{s} is the first stage where $x$ is formed, $x$ is not yet available at \textbf{s} and will only become available at stages above \textbf{s}. However, urelements do not form at stages; when they occur at a stage, they immediately become available at the stage for forming sets. With our reading of $\form$, we can define the availability relation $\avail$.
\begin{definition}\
\begin{enumerate}
\item $x\prec\mathbf{s}$ abbreviates $\exists\mathbf{r}(\mathbf{r}<\mathbf{s}\land x\form\mathbf{r})$.
\item $x \avail \mathbf{s}$ abbreviates $x\prec\mathbf{s}\lor\bigl(\Ur(x)\land x\form\mathbf{s})$
\end{enumerate}
\end{definition}
\noindent That is, $x$ is available at a stage $\mathbf{s}$ just in case either $x$ occurs at a stage below $\mathbf{s}$, or $x$ is a urelement occurring at $\mathbf{s}$. 

\begin{definition}
\STU (\textit{Stage Theory with Urelements}) consists of the following axioms and schemes.\footnote{In this and all subsequent axiom lists, displayed open formulas are understood as universally closed. Parameters permitted in schemes, whether object or stage parameters, are suppressed.}
\begin{enumerate}
\item (\UrDef) \(\Ur(x)\rightarrow \forall y(y\notin x)\).

\item (\Extensionality) \(\Set(x)\land \Set(y)\land \forall z(z\in x\leftrightarrow z\in y)\rightarrow x=y\).

\item (\OrderAxiom) \(\mathbf{r}<\mathbf{s}<\mathbf{t}\rightarrow \mathbf{r}<\mathbf{t}\).

\item (\Staging) \(\exists\mathbf{s}(x\form\mathbf{s})\).

\item (\Priority) \(x\form\mathbf{s}\rightarrow \forall y\in x\ (y\avail\mathbf{s})\).

\item (\Specification)
$\forall x(\varphi(x)\rightarrow x\avail\mathbf{s})\rightarrow \exists y(\Set(y)\land y\form\mathbf{s}\land \forall x(x\in y\leftrightarrow \varphi(x)))$, where \(\varphi(x)\) is a suitable formula in $\mathcal{L}_\textbf{stage}$.

\end{enumerate}
\end{definition}
\noindent We take the first three axioms as part of the definition of urelements, sets, and stages. \UrDef says that urelements have no members; \Extensionality says that sets are determined by their members; \OrderAxiom says that the below relation is transitive. Notice that \OrderAxiom makes availability persistent: if $x\avail\mathbf{r}$ and $\mathbf{r}<\mathbf{s}$, then $x\avail\mathbf{s}$. The last three axioms are substantial but almost literal translations of the Basic Iterative Story. \Staging says that every object occurs at some stage. \Priority says that every member of a set formed at a stage must be available there. \Specification says that any objects available at a stage form a set at that stage. We will come back to the consequences of \STU in Section \ref{subsec: STU consequences}.\footnote{ Boolos \cite{Boolos1989IterationAgain} does not include any urelements. Moreover, the iterative conception he has in mind also seems to presuppose that the urelements form a set, so the urelements are excluded \textit{for simplicity}. Boolos also assumes two additional axioms about stages, which we will discuss in the next section but not include them in the bare iterative conception. \STU also differs from Button's urelemental version \cite[453--454]{Button2021LevelTheory1}, where he adopts the following axioms:

\begin{itemize}
\item [] (\SpecU)
$$\forall x(\varphi(x)\rightarrow(\Ur(x)\lor x\prec\mathbf{s}))\rightarrow \exists y(\Set(y)\land y\form\mathbf{s}\land \forall x(x\in y\leftrightarrow \varphi(x))).$$
\end{itemize}
\SpecU treats every urelement as available at every stage, which makes \UrSet an immediate consequence. Our \Specification instead permits at a stage only those urelements which are available there. Consequently, our theory permits urelements to enter the iterative process at stages above those at which other urelements enter, or at incomparable stages; see \cref{thm:stu-bizarre-model}. Finally, our \Staging requires every urelement be available at some stage. Literally speaking, this is not part of the Basic Iterative Story. But allowing urelements that are never available will complicate the discussion in some trivial manner, so we adopt \Staging to exclude them. }

\subsection{\LTU}
\STU takes the notion of stage as primitive, but we are interested in how sets are formed rather than what stages are. In other words, as Button \cite[438]{Button2021LevelTheory1} points out, our understanding of the notion of set should not depend on the notion of stage. Set theory should be autonomous enough to be able to express the very idea of the iterative conception using its own terms. Level theory, fully developed by Button \cite{Button2021LevelTheory1} based on the work of Scott \cite{Scott1974AxiomatizingSetTheory} and Potter \cite{Potter2004SetTheoryPhilosophy}, aims to overcome this conceptual defect of the stage approach. In their versions of level theory, urelements are either excluded or assumed to form a set. Here we shall develop a more general version of level theory.

In \ZF, we can naturally interpret stages as the familiar $V_\alpha$ hierarchies. With urelements, the definition of a hierarchy can be naturally generalized in \ZFU. \footnote{\ZFU, formulated in $\mathcal{L}_\Ur$,  is \ZF modified to allow urelements. The axioms of \ZFU consist of \UrDef, \Extensionality, \Pairing, \Powerset, \UnionAxiom, \InfinityAxiom, \Separation, \Foundation and \Replacement. See \cite{Yao2026AxiomatizationForcing}, where the theory is called $\ZFUR$, for the exact axiomatization. } Given a set $A$ of urelements, we define $V_\alpha(A)$ by transfinite recursion.
\begin{enumerate}
\item $V_0(A)=A$.

\item $V_{\alpha+1}(A)=\Pow(V_\alpha(A))\cup A$.

\item $V_\gamma(A)=\bigcup_{\alpha<\gamma}V_\alpha(A)$ if $\gamma$ is a limit ordinal.
\end{enumerate}
Call every set of the form $V_\alpha(A)$ a \textit{hierarchy}, and every set of urelements is said to be an \textit{initial} hierarchy. \ZFU proves that $V_{\alpha}(A)$ is a set for every $A$ and $\alpha$, and that every object is in some hierarchy.

To formulate the Basic Iterative Story in $\mathcal{L}_\Ur$, the task becomes defining hierarchies without appealing to the notion of ordinals and the tool of transfinite recursion, and this is where the notion of \textit{level} comes in. We start with some basic definitions.

\begin{definition}
$x$ is \textit{transitive} iff $\Set(x) \land \forall y\in x\ \forall z\in y\ z\in x.$
\end{definition}

\begin{definition}
For every set $x$, the \textit{kernel} of $x$
\begin{equation*}
\ker(x)=\{a\mid \Ur(a)\land \forall y(y\text{ is transitive}\land x\subseteq y\rightarrow a\in y)\},
\end{equation*}
if it exists.
\end{definition}
\noindent Intuitively, $\ker(x)$ is the set of urelements that are involved in the process of forming $x$; over a sufficient base theory, $\ker(x)$ is just the set of urelements in \textit{the transitive closure} of $x$.  We shall use the following convention to ease the notation.
\begin{definition}
$x \subseteq y$ abbreviates $\Set(x) \land \Set(y) \land \forall z \in x (z \in y)$.
\end{definition}
\begin{definition}
For every set $x$, the \textit{potentiation} of $x$ is
\begin{equation*}
\pot x=\{y\mid \exists z(y\subseteq z\in x)\}\cup\ker(x),
\end{equation*}
if it exists.
\end{definition}
\noindent That is, $\pot x$ collects every set which is a subset of some set already in $x$ together with all the urelements in $\ker(x)$.
\begin{definition}
$h$ is a \textit{history}, written $\textup{Hist}(h)$, iff
\begin{enumerate}
\item $\Set(h)$.

\item $\forall x\in h(\Set(x) \to x=\pot(x\cap h))$.

\item $\forall x\in h(\Set(x) \to x=\ker(h)\lor \ker(h)\in x)$.
\end{enumerate}
\end{definition}
\noindent The second clause says that each set in the history is obtained by potentiating its $\in$-predecessors in $h$. The third clause is the key for our purpose: the history itself must remember which urelements its set members are built over, and every set in the history is \textit{maximal} in the sense that no urelement is forgotten. A level is then obtained by potentiating a history.

\begin{definition}
$x$ is a \textit{level}, written $\Lev(x)$, iff $\exists h(\textup{Hist}(h)\land x=\pot h).$
\end{definition}
\noindent Every set of urelements is a history for the vacuous reason and hence a level. As we will see, sets in a history are levels with a fixed kernel.\footnote{In the pure setting of Scott \cite[214]{Scott1974AxiomatizingSetTheory}, Montague \cite[139]{Montague1965SetTheoryHigherOrderLogic}, and Button \cite{Button2021LevelTheory1}, there are no urelements, so every kernel is empty and potentiation reduces to $\{y\mid \exists z(y\subseteq z\in x)\}.$ The fixed-kernel clause for histories is then redundant. Button defined levels with urelements in \cite[Appendix A]{Button2021LevelTheory1} by assuming  \UrSet, which are levels in our sense whose kernel is the set of all urelements.}

\begin{definition}
\LTU (\textit{Level Theory with Urelements}) consists of the following axioms and schemes in $\mathcal{L}_\Ur$.
\begin{enumerate}
\item (\UrDef) \(\Ur(x)\rightarrow \forall y(y\notin x)\).

\item (\Extensionality) \(\Set(x)\land \Set(y)\land \forall z(z\in x\leftrightarrow z\in y)\rightarrow x=y\).

\item (\Separation) \(\Set(x)\rightarrow \exists y(\Set(y)\land \forall z(z\in y\leftrightarrow z\in x\land \varphi(z)))\), where \(\varphi\) is a suitable formula in $\mathcal{L}_\Ur$.

\item {\small (\Stratification) \(\bigl(\Set(x)\rightarrow\exists y(\Lev(y)\land x\subseteq y)\bigr)\land\bigl(\Ur(x)\rightarrow\exists y(\Lev(y)\land x\in y)\bigr)\).}
\end{enumerate}
\end{definition}

\noindent The idea is that things that are available at a level are precisely the members of the level, and a set is formed on at level whenever it is a subset of that level. So \Separation and \Stratification provide a level-theoretic formalization of the Basic Iterative Story, where levels are intended for stages.
\begin{definition}\
\begin{enumerate}
\item $x$ is \textit{potent}, written $\textup{Potent}(x)$, iff
$\Set(x)\land\forall y \forall z (y \subseteq z \in x \rightarrow y \in x).$ 

\item $x$ is \textit{supertransitive} iff $x$ is transitive and potent.
\end{enumerate}
\end{definition}
\noindent For every set $x$, $\pot(x)$ is potent if it exists. Instead of working in \LTU, in the next subsection we will work in the following weaker theory, which will be shown to be a subtheory of both \LTU and \STU.
\begin{definition}
The theory \LTUMinus consists of \UrDef, \Extensionality, \Separation, and the following weak stratification principle.

\begin{itemize}
\item [] (\StratificationMinus) $\Set(x) \rightarrow \exists y(y \text{ is supertransitive} \land x \subseteq y).$
\end{itemize}
\end{definition}

We verify that \LTU indeed proves \LTUMinus by showing that every level is supertransitive.

\begin{lemma}\label{lem:levels-potent-transitive}
Every level is supertransitive.
\end{lemma}

\begin{proof}
Let $s$ be a level, so $s=\pot(h)$ for some history $h$. Since $h$ is a set, $s$ is potent. To see that $s$ is transitive, let $y$ be a set in $s$.Then there is some set $x \in h$ such that $y \subseteq x$. $x=\pot(x \cap h)$ and so we have $y \subseteq \pot(x \cap h) \subseteq \pot(h)=s.$ Thus $s$ is transitive.
\end{proof}
\subsection{Consequences of \texorpdfstring{\LTUMinus}{\LTU-}}

Now we establish some basic facts about level theory in \LTUMinus. The arguments in Lemma 18 - 24 are analogous to those given in Button \cite{Button2021LevelTheory1}. We start with the existence of kernels and potentiations, followed by some basic facts about kernels which will be used repeatedly.
\begin{prop}[\LTUMinus]\label{prop:kernel-pot-existence}
For every set $x$, $\ker(x)$ and $\pot(x)$ are sets.
\end{prop}
\begin{proof}
By \StratificationMinus, there is a supertransitive set $t$ such that $x\subseteq t$. Then $\ker(x)$ is contained in $t$ and hence a set by \Separation on $t$. If $y\subseteq z\in x$ for some $z$, then $z\in t$ so the potency of $t$ gives $y\in t$. \Separation on $t$ gives $\pot(x)$ as a set.
\end{proof}
\begin{prop}[\LTUMinus]\label{prop:kernelbasics}
For all sets $x$ and $y$:
\begin{enumerate}
\item If $x \subseteq y$, then $\ker(x) \subseteq \ker(y)$.

\item If $x \in y$, then $\ker(x) \subseteq \ker(y)$.

\item $\{a \in x \mid \Ur(a)\} \subseteq \ker(x)$.

\item For every urelement $a$,
$$a \in \ker(x) \leftrightarrow a \in x \lor \exists y \in x(\Set(y)\land a \in \ker(y)).$$
\end{enumerate}
\end{prop}
\begin{proof}
(1) Let $a \in \ker(x)$, and let $z$ be a transitive set with $y \subseteq z$. We have $x \subseteq z$, and so $a \in z$. Thus $a \in \ker(y)$.

(2) Let $a \in \ker(x)$, and let $z$ be a transitive set with $y \subseteq z$. We have $x \in z$; so $x \subseteq z$ and hence $a \in z$. Thus $a \in \ker(y)$.

(3) Let $a \in x$ be a urelement, and let $z$ be a transitive set with $x \subseteq z$. Then $a \in z$. Thus $a \in \ker(x)$.

(4) The right-to-left direction follows from (2) and (3). For the converse, suppose that $a \in \ker(x)$, $a \notin x$, and there is no set $y \in x$ such that $a \in \ker(y)$. By \StratificationMinus, fix a supertransitive set $z$ such that $x \subseteq z$. By \Separation, let
$$z^*=\{w \in z \mid w \neq a \land (\Set(w) \rightarrow a \notin \ker(w))\}.$$
We claim that $z^*$ is transitive. Suppose $u \in w \in z^*$. Since $z$ is transitive, $u \in z$. Also $u \neq a$, since otherwise $a \in w$, and hence $a \in \ker(w)$ by (3). If $\Set(u)$, then $\ker(u) \subseteq \ker(w)$ by (2), and therefore $a \notin \ker(u)$. Thus $u \in z^*$.

Moreover, $x \subseteq z^*$: if $w \in x$, then $w \neq a$ by assumption, and if $\Set(w)$, then $a \notin \ker(w)$ by assumption. So $z^*$ is a transitive set with $x \subseteq z^*$ and $a \notin z^*$, contradicting $a \in \ker(x)$.
\end{proof}

\begin{lemma}[\LTUMinus]\label{lem:pot-kernel}
If $x=\pot(y)$, then $\ker(x)=\ker(y)$.
\end{lemma}

\begin{proof}
First, $y\subseteq x$. For if $w\in y$ and $\Set(w)$, then $w\subseteq w\in y$, so $w\in \pot(y)=x$; and if $w$ is a urelement, then $w\in \ker(y)\subseteq \pot(y)=x$. Hence $\ker(y)\subseteq \ker(x)$. Conversely, let $a\in\ker(x)$. By \cref{prop:kernelbasics}(4), either $a\in x$, or $a\in\ker(z)$ for some set $z\in x$. In the first case, we have $a\in\ker(y)$. In the second case, fix such a set $z$ and some $w\in y$ with $z\subseteq w$. So $\ker(z)  \subseteq \ker(w) \subseteq \ker(y)$ by \cref{prop:kernelbasics} and hence $a \in \ker(y)$.\end{proof}

The following is a schematic version of Button \cite[Lemma 3.5]{Button2021LevelTheory1}, which is due to Scott, and the proof is included for completeness (which only uses \Separation and \Extensionality). We will refer to this lemma as \textit{Potent Induction}: every non-empty class of potent sets has an $\in$-minimal element.
\begin{lemma}[\LTUMinus]\label{lem:potent-minimality}
For every formula $\varphi(x)$, the following scheme holds:
$$\bigl(\exists x\,\varphi(x)\land\forall x(\varphi(x)\rightarrow\textup{Potent}(x))\bigr)\rightarrow\exists w\bigl(\varphi(w)\land\forall x(\varphi(x)\rightarrow x\notin w)\bigr).$$
\end{lemma}

\begin{proof}
Suppose the antecedent holds, and fix $w_0$ such that $\varphi(w_0)$. Let
$$y=\{x\in w_0\mid\forall w(\varphi(w)\rightarrow x\in w)\},\qquad z=\{x\in y\mid x\notin x\}.$$
Since $\varphi(w_0)$, $y = \{x \mid\forall w(\varphi(w)\rightarrow x\in w)\}$. Moreover, $z\notin y$ by Russell's Paradox. Hence there is some $w$ such that $\varphi(w)$ and $z\notin w$. If $\varphi(x)$, then $z\subseteq y\subseteq x$. Since $w$ is potent, $x\in w$ would imply $z\in w$, a contradiction. Thus no $\varphi$-object belongs to $w$, as required.
\end{proof}

We proceed to prove some expected properties of histories and levels: every history is an initial sequence of levels with a fixed kernel. This is proved in Button \cite{Button2021LevelTheory1} for the pure-set case.
\begin{lemma}[\LTUMinus]\label{lem:history-members-transitive}
Every set in a history is transitive.
\end{lemma}

\begin{proof}
Fix a history $h$ and suppose that some set in $h$ is not transitive. Every set $x\in h$ satisfies $x=\pot(x\cap h)$ and is therefore potent. By Potent Induction, fix an $\in$-minimal non-transitive set $x\in h$. Every set in $x\cap h$ is therefore transitive.

Let $y \in z \in x$. By \UrDef, $z$ is a set. Since $x=\pot(x \cap h)$, there is some $w \in x \cap h$ such that $z \subseteq w$. By hypothesis, $w$ is transitive, so $y \in w$. If $\Set(y)$, then $y \subseteq w$, so $y \in \pot(x \cap h)=x$. If $y=a$ is a urelement, then $a \in \ker(w) \subseteq \ker(x \cap h)$, so again $a \in \pot(x \cap h)=x$. Thus $x$ is transitive, contradicting the definition of $x$.
\end{proof}

\begin{lemma}[\LTUMinus]\label{lem:history-kernel}
For every history $h$, either $h=\ker(h)$ or $\ker(h)\in h$. Hence, for every level $s$, either $s = \ker(s)$ or $\ker(s) \in s$.
\end{lemma}

\begin{proof}
Let $A=\ker(h)$, and suppose $A\notin h$. Then $h$ has no set member. Otherwise, every set in $h$ is potent, so by Potent Induction, there is an $\in$-minimal set $x\in h$. Since $A\notin h$, $x\neq A$, and hence $A\in x$. But $x=\pot(x\cap h)$; since $A$ is a set, there is a set $y\in x\cap h$ such that $A\subseteq y$, contradicting the minimality of $x$. Thus $h$ is a set of urelements, and so $h=\ker(h)=A$.

For the consequence, let $s=\pot(h)$ be a level. By \cref{lem:pot-kernel}, $\ker(s)=\ker(h)$. By the first part, either $h= \ker(h)$ or $\ker(h) \in h$. In the first case, $s =\pot(h)= h = \ker(s)$ since $h$ has no set members. In the second case, $ker (s) = \ker(h) \in h \subseteq s$, and therefore $\ker(s)\in s$.
\end{proof}
\begin{lemma}[\LTUMinus]\label{lem:history-members-kernel}
Let $h$ be a history. For every set $x\in h$, $\ker(x)=\ker(h)$.
\end{lemma}

\begin{proof}
Let $x\in h$ be a set. So we have $\ker(x)\subseteq\ker(h)$. If $x=\ker(h)$, $\ker(h) \subseteq x = \ker(x)$. If $\ker(h)\in x$, then $\ker(h) = \ker (\ker(h)) \subseteq \ker(x)$. Thus, $\ker(x) = \ker(h)$.
\end{proof}

\begin{lemma}[\LTUMinus]\label{lem:history-members-levels}
Every set in a history is a level.
\end{lemma}

\begin{proof}
Let $h$ be a history and $x\in h$ be a set. Since $x=\pot(x\cap h)$, it is enough to show that $x\cap h$ is a history.

Let $y\in x\cap h$ be a set. Then $y=\pot(y\cap h)$. Since $x$ is transitive by \cref{lem:history-members-transitive}, $y\subseteq x$, and hence $y\cap h=y\cap(x\cap h)$. Thus $y=\pot(y\cap(x\cap h)).$

To check the kernel clause, let $y \in x \cap h$ be a set. Since $h$ is a history, $y=\ker(h)$ or $\ker(h)\in y$. By \cref{lem:pot-kernel}, $\ker(x)=\ker(x\cap h)$, and by \cref{lem:history-members-kernel}, $\ker(x)=\ker(h)$. So $y=\ker(x \cap h)$ or $\ker(x \cap h)\in y$. Therefore, $x \cap h$ is indeed a history.
\end{proof}
\noindent The next lemma says that every level can be decomposed into lower levels with the same kernel.
\begin{lemma}[\LTUMinus]\label{lem:level-decomposition}
For every level $s$, let
$$L_s=\{r\in s\mid\textup{Lev}(r)\land\ker(r)=\ker(s)\}.$$
Then $s=\ker(s)\cup\pot(L_s).$
\end{lemma}

\begin{proof}
Let $x\in\ker(s)\cup\pot(L_s)$. If $x\in\ker(s)$, then $x\in s$ since $s$ is transitive. Suppose $x\in\pot(L_s)$. If $x$ is a set, then $x\subseteq r\in L_s\subseteq s$ for some $r$, and hence $x\in s$ by the potency of $s$. If $x$ is a urelement, then $x\in\ker(L_s)\subseteq\ker(s)\subseteq s$.

Conversely, let $x\in s$. If $x$ is a urelement, then $x\in\ker(s)$. Suppose $x$ is a set, and fix a history $h$ such that $s=\pot(h)$. There is a set $r \in h \subseteq s$ such that $x\subseteq r$. By \cref{lem:history-members-levels}, $r$ is a level, and by \cref{lem:history-members-kernel,lem:pot-kernel}, $\ker(r)=\ker(h)=\ker(s)$. Thus $r\in L_s$, and hence $x\in\pot(L_s)$.
\end{proof}

Button \cite[Lemma 3.9]{Button2021LevelTheory1} proves that when there are no urelements, levels are linearly ordered, and hence well-ordered, by the membership relation. This is, of course, not true in \LTU since any two different sets of urelements are incomparable levels. But using Button's argument, we can show that levels with the same kernel are indeed linearly ordered.
\begin{lemma}[\LTUMinus]\label{lem:same-kernel-levels-comparable}
If $s$ and $t$ are levels with the same kernel, then $s\in t$, $s=t$, or $t\in s$. Consequently, levels with the same kernel are well-ordered by $\in$.
\end{lemma}

\begin{proof}
Let $A$ be a set of urelements. Two levels with kernel $A$ are \textit{incomparable} if neither belongs to the other and they are not equal. Suppose, for a contradiction, that there are incomparable levels with kernel $A$. Since every level is potent, by Potent Induction, there is an $\in$-minimal level $s$ with kernel $A$ which is incomparable with some level with kernel $A$; fix such an $s$. By a second application of Potent Induction, fix an $\in$-minimal level $t$ with kernel $A$ which is incomparable with $s$.

We show that $s\subseteq t$. Let $x\in s$. If $x$ is a urelement, then $x\in\ker(s)=A=\ker(t)$, and since $t$ is transitive, $x\in t$. Suppose $\Set(x)$. By \cref{lem:level-decomposition}, there is a level $r\in s$ such that $\ker(r)=A$ and $x\subseteq r$. By the minimality of $s$, the levels $r$ and $t$ are comparable. If $r=t$ or $t\in r$, then $t\in s$, contradicting the incomparability of $s$ and $t$. Hence $r\in t$. Since $t$ is potent, $x\in t$. The same argument, using the minimality of $t$, gives $t\subseteq s$. So $s=t$, again contradicting the assumption. Therefore any two levels with the same kernel are comparable.
\end{proof}

We conclude this subsection with \textit{the Level-Shrink Lemma}: any level can be shrunk to a level with a smaller kernel that contains every set in the original level whose kernel is small. We first isolate the general fact behind several later constructions of histories with a fixed kernel.

\begin{lemma}[\LTUMinus]\label{lem:fixed-kernel-history}
Let $A$ be a set of urelements and let $h$ be a set satisfying
\begin{enumerate}
\item $A = \{a \in h \mid \Ur(a)\}$;

\item every set $l\in h$ is a level with $\ker(l) = A$; and

\item if $q \in l\in h$ and $q$ is a level with $\ker(q) = A$, then $q\in h$.
\end{enumerate}
Then $h$ is a history and $\ker(h)=A$.
\end{lemma}

\begin{proof}
Since $A\subseteq h$, \cref{prop:kernelbasics}(3) gives $A\subseteq\ker(h)$. Conversely, if $a\in\ker(h)$, then by \cref{prop:kernelbasics}(4), either $a\in h$, in which case $a\in A$ by (1), or $a\in\ker(l)$ for some set $l\in h$, in which case $a\in A$ by (2). Thus $\ker(h)=A$.

Let $l\in h$ be a level. We show that $l=\pot(l\cap h)$. It follows from conditions (1)--(3) that 
$$l\cap h=\{q\in l\mid \Lev(q)\land\ker(q)=A\}\cup A.$$
\noindent Let $x\in l$. If $x$ is a urelement, then $x\in\ker(l)=A\subseteq\ker(l\cap h)$, so $x\in\pot(l\cap h)$. If $x$ is a set, then by \cref{lem:level-decomposition}, there is a level $q\in l$ such that $x\subseteq q$ and $\ker(q)=A$. By (3), $q\in h$, and hence $x\in\pot(l\cap h)$. Thus $l\subseteq\pot(l\cap h)$. Conversely, let $x\in\pot(l\cap h)$. If $x$ is a set, then $x\subseteq q\in l\cap h$ for some set $q$, so the potency of $l$ gives $x\in l$. If $x$ is a urelement, then $x\in\ker(l\cap h)$ so $x\in l$. Therefore $\pot(l\cap h)\subseteq l$ and hence $l=\pot(l\cap h)$.
Finally, \cref{lem:history-kernel} gives $l=A$ or $A\in l$. Since $A=\ker(h)$, the kernel clause also holds. Therefore $h$ is a history.
\end{proof}

\begin{namedresult*}{Level-Shrink Lemma}[\LTUMinus]\phantomsection\label{lem:level-shrink}
For every set $A$ of urelements and every level $s$ with $A\subseteq\ker(s)$, let
$$s^A=\{x\in s\mid \Set(x)\land \ker(x)\subseteq A\}\cup A.$$
Then $s^A$ is a level with $\ker(s^A) = A$.
\end{namedresult*}

\begin{proof}
A simple kernel calculation using \cref{prop:kernelbasics} shows that $\ker(s^A)= A$. Suppose \textit{for reductio} that $s^A$ is not a level. Since every level is potent, by Potent Induction, fix an $\in$-minimal level $s$ such that $A\subseteq\ker(s)$ and $s^A$ is not a level. Let
$$h=\{r\in s\mid \textup{Lev}(r)\land \ker(r)=A\}\cup A.$$
By \StratificationMinus, there is a supertransitive set containing $s$. Since $A\subseteq\ker(s)$, this set also contains $A$, so \Separation gives $h$ as a set. Its urelement members are exactly the members of $A$, and every set in $h$ is a level with kernel $A$. Moreover, if $q\in r\in h$ and $q$ is a level with kernel $A$, then $r\in s$ and the transitivity of $s$ gives $q\in s$, so $q\in h$. By \cref{lem:fixed-kernel-history}, $h$ is a history and $\ker(h)=A$. We show that $s^A=\pot(h)$.

First let $x\in\pot(h)$. If $x$ is a set, then $x\subseteq r\in h$ for some $r$. Since $r\in s$ and $s$ is potent, $x\in s$; and since $\ker(r)=A$, $\ker(x)\subseteq A$. Hence $x\in s^A$. If $x$ is a urelement, then $x\in\ker(h)=A$, so $x\in s^A$.

Conversely, let $x\in s^A$. If $x$ is a urelement, then $x\in A=\ker(h)\subseteq\pot(h)$. Suppose $\Set(x)$. Then $x\in s$ and $\ker(x)\subseteq A$. By \cref{lem:level-decomposition}, there is a level $r\in s$ such that $x\subseteq r$ and $\ker(r)=\ker(s)$. Fix such an $r$. Since $A\subseteq\ker(r)$ and $r\in s$, the minimality of $s$ gives that $r^A$ is a level, and $\ker(r^A)=A$. Also $r^A\subseteq r\in s$, so $r^A\in s$ by the potency of $s$, and hence $r^A\in h$. It remains only to see that $x\subseteq r^A$. If $y\in x$ is a set, then $y\in r$ and $\ker(y)\subseteq\ker(x)\subseteq A$, so $y\in r^A$. If $y=a$ is a urelement, then $a\in\ker(x)\subseteq A$, so $a\in r^A$. Thus $x\subseteq r^A\in h$, and hence $x\in\pot(h)$. 

Therefore $s^A=\pot(h)$, which means $s^A$ is a level after all, contradicting the assumption about $s$.
\end{proof}

\subsection{\ZFU proves \LTU}
We verify that our definition of level coincides with the notion of hierarchy in \ZFU. Since \ZFU proves that every set is contained in a hierarchy and every hierarchy is supertransitive, \ZFU proves \StratificationMinus and thus \LTUMinus. So in the following argument we may freely use the established consequences of \LTUMinus.
\begin{theorem}[\ZFU]\label{thm:levels-exactly-hierarchies}
$x$ is a level if and only if $x$ is a hierarchy.
\end{theorem}
\begin{proof}
To show that every hierarchy is a level, for every set $A$ of urelements and ordinal $\alpha$, define $h_\alpha (A) = A\cup\{V_\beta(A)\mid\beta<\alpha\}$. It is easy to check that $V_\alpha (A) = \pot (h_\alpha (A))$. So for every $V_\beta(A) \in h_\alpha (A)$, $V_\beta (A) = \pot (h_\beta (A)) = \pot (V_\beta (A) \cap h_\alpha(A))$. Since every non-initial hierarchy contains $A$, it follows that $h_\alpha(A)$ is a history and hence $V_\alpha(A)$ is a level.

Conversely, let $s =\pot(h)$ be a non-initial level, and put $A=\ker(h)=\ker(s)$. By \cref{lem:history-members-kernel,lem:history-members-levels}, the set members of $h$ are levels with kernel $A$; so by \cref{lem:same-kernel-levels-comparable}, they are well-ordered by $\in$ with some order type $\alpha$. So we can enumerate the levels in $h$ as $\{l_\beta \mid \beta < \alpha\}$. Since $s$ is non-initial, $h\neq A$ so $A\in h$ and hence $A=\pot(A\cap h) = A\cap h$. Thus $A\subseteq h$. We therefore have
$$l_\beta\cap h=A\cup\{l_\gamma\mid\gamma<\beta\},$$
for every $\beta < \alpha$. Assuming that $l_\gamma = V_\gamma(A)$ for every $\gamma < \beta$, we have 
$$l_\beta = \pot(l_\beta \cap h) = \pot(A\cup\{V_\gamma(A)\mid\gamma<\beta\}) = \pot (h_\beta(A)) = V_\beta (A).$$
So it follows by induction that $l_\beta=V_\beta(A)$ for all $\beta < \alpha$. Consequently,
$$h=A\cup\{V_\beta(A)\mid\beta<\alpha\} = h_\alpha(A)$$
and so $s =V_\alpha(A)$. Thus every level is a hierarchy.\end{proof}

\noindent Consequently, \LTU is a subtheory of \ZFU.  

We note that \ZFU, and hence \LTU, has many models in which \UrSet fails. Somewhat surprisingly, the urelements may form a proper class without being ``large'', e.g., there are models in which every initial level is finite (Theorem \ref{thm:zfu-not-strongly-quasi-categorical}). And for any infinite cardinal $\kappa$, there are models in which every initial level has cardinality at most $\kappa$, even though no set contains all the urelements; see Yao \cite[Theorem 2.17]{Yao2026AxiomatizationForcing}. \UrSet can also fail when the class of urelement is large. Yet intermediate notions of largeness come apart: there are models in which every pure set is equinumerous to an initial level, while some impure set is not (\cite[Theorem 37]{Yao2025Plenitude} ). Furthermore, second-order \ZFU has models where the urelements are strictly more numerous than the pure sets (see \cite{HamkinsYao2024AbundantUrelements}). In view of Theorem \ref{thm:levels-exactly-hierarchies} and the fact that \LTU interprets \STU (Theorem \ref{thm:ltu-interprets-stu}), these results exhibit many ways of telling the Basic Iterative Story in which neither \UrStage nor \UrSet holds.\footnote{Button \cite[Appendix B]{Button2021LevelTheory1} considers one such story, motivated by Limitation of Size, in which the urelements are as numerous as all objects. We take no stand on whether this is the most natural iterative story when \UrSet fails. Our point is only that it is not the only one.} Taking \UrStage to be part of the iterative conception would therefore rule out all these stories as illegitimate with no philosophical justification.

\section{Comparison}\label{sec: Comparison}

\subsection{Consequences of \STU}\label{subsec: STU consequences}
We now compare \STU with \LTU by first showing that \STU proves \LTUMinus. One key notion in stage theory is \textit{slice} defined in \cite{Button2021LevelTheory1}.

\begin{definition}\label{def:stu-slice}
For any stage $\mathbf{s}$, the $\mathbf{s}$-\textit{slice} is
$\check{\mathbf{s}}=\{x\mid x\avail\mathbf{s}\}.$ An object $x$ is a \textit{slice}, written $\Slice(x)$, iff $x=\check{\mathbf{s}}$ for some stage $\mathbf{s}$.\footnote{ In the pure setting this coincides with Button's definition of a slice, since sets are available at a stage exactly when they were formed at a stage below it. With urelements, the slice at $\mathbf{s}$ also contains the urelements which become available at $\mathbf{s}$ itself.}
\end{definition}

\begin{lemma}[\STU]\label{lem:stu-slice-properties}
Let $\mathbf{s}$ be a stage.
\begin{enumerate}
\item $\check{\mathbf{s}}$ is a set and $\check{\mathbf{s}} \form \mathbf{s}$.

\item For every set $x$, $x\form \mathbf{s}\leftrightarrow x\subseteq \check{\mathbf{s}}.$

\item $\check{\mathbf{s}}$ is supertransitive.

\end{enumerate}
\end{lemma}

\begin{proof}
(1) Applying \Specification at $\mathbf{s}$ to the formula $x\avail\mathbf{s}$ gives the set $\check{\mathbf{s}}$, and $\check{\mathbf{s}}\form\mathbf{s}$.

(2) If $x\form\mathbf{s}$, then every member of $x$ is available at $\mathbf{s}$ by \Priority, so $x\subseteq\check{\mathbf{s}}$. Conversely, suppose that $x\subseteq\check{\mathbf{s}}$. Then every member of $x$ is available at $\mathbf{s}$. Applying \Specification at $\mathbf{s}$ to the formula $y\in x$ gives the set $x$, and $x\form\mathbf{s}$.

(3) First suppose $y\in x\in\check{\mathbf{s}}$. Since $x$ is a set available at $\mathbf{s}$, $x\form\mathbf{r}<\mathbf{s}$ for some stage $\mathbf{r}$. By \Priority, $y$ is available at $\mathbf{r}$ and hence available at $\mathbf{s}$. Hence $y\in\check{\mathbf{s}}$, so $\check{\mathbf{s}}$ is transitive. Now suppose $x\subseteq y\in\check{\mathbf{s}}$. Then $y$ is a set such that $y \form \mathbf{r}<\mathbf{s}$ for some stage $\mathbf{r}$. So $x \subseteq \check{\mathbf{r}}$ by \Priority. Hence (2) gives $x\form\mathbf{r}$, so $x\prec\mathbf{s}$ and therefore $x\in\check{\mathbf{s}}$. Thus $\check{\mathbf{s}}$ is potent, and hence supertransitive.
\end{proof}

\begin{theorem}\label{thm:stu-proves-ltu-minus}
\STU proves \LTUMinus.
\end{theorem}

\begin{proof}
\UrDef and \Extensionality are axioms of \STU. To prove \Separation, fix a set $w$. By \Staging, there is a stage $\mathbf{s}$ such that $w\form\mathbf{s}$. By \Priority, every member of $w$ is available at $\mathbf{s}$. Hence \Specification at $\mathbf{s}$, applied to the formula $x\in w\land\varphi(x)$, gives the required subset of $w$. For \StratificationMinus, let $x$ be a set. By \Staging, there is a stage $\mathbf{s}$ such that $x\form \mathbf{s}$. By \cref{lem:stu-slice-properties}(2), $x\subseteq\check{\mathbf{s}}$; and by \cref{lem:stu-slice-properties}(3), $\check{\mathbf{s}}$ is supertransitive. Thus every set is contained in a supertransitive set.
\end{proof}

The following lemma about slice decomposition will be useful later on.
\begin{lemma}[\STU]\label{lem:stu-slice-decomposition}
Let $\mathbf{s}$ be a stage. Then $$\check{\mathbf{s}}=\pot(H_{\mathbf{s}})\cup\ker(\check{\mathbf{s}}),$$ where $H_{\mathbf{s}}=\{x\in\check{\mathbf{s}}\mid \Slice(x)\}.$ 
\end{lemma}
\begin{proof}
First suppose $x\in\check{\mathbf{s}}$. If $x$ is a urelement, then $x\in\ker(\check{\mathbf{s}})$. If $x$ is a set, then $x\form\mathbf{r}<\mathbf{s}$ for some stage $\mathbf{r}$, and so $x\subseteq\check{\mathbf{r}}$ by \cref{lem:stu-slice-properties}(2). Since $\check{\mathbf{r}}\form\mathbf{r}$ and $\mathbf{r}<\mathbf{s}$, we have $\check{\mathbf{r}}\in\check{\mathbf{s}}$; hence $\check{\mathbf{r}}\in H_{\mathbf{s}}$, and $x\in\pot(H_{\mathbf{s}})$.

Conversely, let $x\in\pot(H_{\mathbf{s}})\cup\ker(\check{\mathbf{s}})$. If $x\in\ker(\check{\mathbf{s}})$, then $x\in\check{\mathbf{s}}$ since $\check{\mathbf{s}}$ is transitive. If $x\in\pot(H_{\mathbf{s}})$ and $x$ is a set, then $x\subseteq y\in H_{\mathbf{s}}\subseteq\check{\mathbf{s}}$ for some $y$, and so $x\in\check{\mathbf{s}}$ by the potency of $\check{\mathbf{s}}$. If $x\in\pot(H_{\mathbf{s}})$ and $x$ is a urelement, then $x\in\ker(H_{\mathbf{s}})$; since $H_{\mathbf{s}}\subseteq\check{\mathbf{s}}$ and $\check{\mathbf{s}}$ is transitive, $x\in\check{\mathbf{s}}$.
\end{proof}

\begin{theorem}[\STU]\label{thm:stu-foundation}
The axiom of \Foundation holds, i.e.,
$$\Set(x)\land \exists y(y\in x)\rightarrow \exists y\in x\forall z\in y(z\notin x).$$
\end{theorem}

\begin{proof}
Let $x$ be a non-empty set. Fix a stage $\mathbf{s}$ such that $x\form \mathbf{s}$. $x\subseteq\check{\mathbf{s}}$, so $x\cap\check{\mathbf{s}}$ is non-empty. The class of slices that intersect $x$ is then non-empty, and every slice is potent by \cref{lem:stu-slice-properties}. Hence, by Potent Induction, there is an $\in$-minimal slice $\check{\mathbf{t}}$ such that $x\cap\check{\mathbf{t}}\neq\emptyset$. Fix $y\in x\cap\check{\mathbf{t}}$, and we may assume $y$ is a set. Let $z\in y$. Since $y\notin\ker(\check{\mathbf{t}})$, \cref{lem:stu-slice-decomposition} gives $y\in\pot(H_{\mathbf{t}})$. Thus $y\subseteq q\in H_{\mathbf{t}}$ for some slice $q$. Hence $z\in q\in\check{\mathbf{t}}$. If $z\in x$, then $q$ is a slice which intersects $x$ and $q\in\check{\mathbf{t}}$, contradicting the minimality of $\check{\mathbf{t}}$. Therefore, $y$ is $\in$-minimal in $x$, as desired.
\end{proof}

\subsection{\LTU Interprets \STU}

Next, we show that \LTU proves whatever \STU proves about sets and urelements. The motivation behind \LTU suggests the following translation.
\begin{definition}\label{def:canonical-translation}
The \textit{canonical translation} $(-)^*$ from $\mathcal{L}_\textbf{stage}$ to $\mathcal{L}_\Ur$ leaves object variables, membership, and urelementhood unchanged, commutes with the Boolean connectives and object quantifiers, and has the following non-trivial clauses:
\begin{align*}
  \Stage(\mathbf{s})^*&\equiv \Lev(\mathbf{s}),\\
(\mathbf{r}<\mathbf{s})^*&\equiv \Lev(\mathbf{r})\land\Lev(\mathbf{s})\land\mathbf{r}\in\mathbf{s},\\
(x\form\mathbf{s})^*&\equiv\bigl(\Set(x)\land x\subseteq\mathbf{s}\bigr)\lor\bigl(\Ur(x)\land x\in\mathbf{s}\bigr),\\
(\forall\mathbf{s}\,\varphi)^*&\equiv\forall\mathbf{s}(\Stage(\mathbf{s})^*\rightarrow\varphi^*).
\end{align*}
\end{definition}
\noindent Namely, levels are intended for stages; a set occurs at a stage by being a subset of the stage, while a urelement occurs at a stage by being a member of it. 

\begin{theorem}\label{thm:ltu-interprets-stu}
For every $\mathcal{L}_\Ur$-sentence $\varphi$, \LTU$\vdash \varphi$ if \STU$\vdash \varphi$.
\end{theorem}

\begin{proof}
Since the canonical translation does not change any $\mathcal{L}_\Ur$-formula, it suffices to show that \LTU $\vdash \varphi^*$ for every axiom $\varphi$ of \STU. First, \UrDef and \Extensionality are unchanged, and $\mathsf{Order}^*$ holds in \LTU because levels are transitive. For $\mathsf{Staging}^*$, if $x$ is a set, \Stratification gives a level $\mathbf{s}$ such that $x\subseteq\mathbf{s}$; if $x$ is a urelement, it gives a level $\mathbf{s}$ such that $x\in\mathbf{s}$. In either case, $(x\form\mathbf{s})^*$ holds.

Next, we verify that, under the canonical translation, the objects available at a stage are exactly the members of the corresponding level. That is,
$$\bigl(x\avail\mathbf{s}\bigr)^*\quad\longleftrightarrow\quad x\in\mathbf{s}.$$
For the forward direction, suppose first that $x\prec\mathbf{s}$. There is a level $\mathbf{r}\in\mathbf{s}$ such that $x\form\mathbf{r}$. If $x$ is a set, then $x\subseteq\mathbf{r}\in\mathbf{s}$, so $x\in\mathbf{s}$ by the potency of $\mathbf{s}$. If $x$ is a urelement, then $x\in\mathbf{r}\in\mathbf{s}$, so $x\in\mathbf{s}$ by transitivity. The remaining case of availability is that $x$ is a urelement and $x\form\mathbf{s}$, which translates directly as $x\in\mathbf{s}$. Conversely, suppose $x\in\mathbf{s}$. If $x$ is a urelement, then $x\form\mathbf{s}$ after translation, so $x\avail\mathbf{s}$. If $x$ is a set, \cref{lem:level-decomposition} gives a level $\mathbf{r}\in\mathbf{s}$ such that $x\subseteq\mathbf{r}$. Hence $x\form\mathbf{r}<\mathbf{s}$, so again $x\avail\mathbf{s}$.

For $\mathsf{Priority}^*$, suppose $(x\form\mathbf{s})^*$ and $y\in x$. Then $x$ is a set and $x\subseteq\mathbf{s}$, so $y\in\mathbf{s}$ and hence $y\avail\mathbf{s}$. For $\mathsf{Specification}^*$, suppose every object satisfying $\varphi(x)$ is available at $\mathbf{s}$ under the translation. Then every such object belongs to $\mathbf{s}$. By \Separation, $y=\{x\in\mathbf{s}\mid\varphi(x)\}$ is a subset of $\mathbf{s}$. Therefore, $y\form\mathbf{s}$ after translation, and $y$ has exactly the required members.
\end{proof}

\subsection{Non-Equivalence}

Finally, we show that \STU proves less about sets and urelements than \LTU by constructing a model of \STU where \Stratification fails. Although one can construct trivial finite countermodels as such, our model has infinite stages in order to show that finiteness is not the reason why these two theories come apart.

\begin{theorem}\label{thm:stu-bizarre-model}
\STU does not prove \Stratification.
\end{theorem}

\begin{proof}
Let $A=\{a_n\mid n<\omega\}$ be a countably infinite set of urelements. The model will have an $\omega + \omega$-sequence of stages. For each $n<\omega$, recursively define
$$\mathbf{s}_n=\mathcal{P}\left(\bigcup_{i<n}\mathbf{s}_i\cup\{a_n\}\right)\cup\{a_n\}.$$
Thus, at $\mathbf{s}_n$, the new urelement $a_n$ occurs and thus becomes available; then every set of available objects is formed. For every $\omega\leq\alpha<\omega+\omega$ recursively define
$$\mathbf{s}_\alpha=\mathcal{P}\left(\bigcup_{\beta<\alpha}\mathbf{s}_\beta\right).$$
So after all the finite stages, no new urelement occurs and new sets are formed iteratively. Define the object domain $M$ as $\bigcup_{\alpha<\omega+\omega}\mathbf{s}_\alpha$, the stage domain $\mathcal{S}$ as $\{\mathbf{s}_\alpha\mid\alpha<\omega+\omega\}$ and the class of urelements $\Ur$ as $A$. The interpretation of $\in$ is inherited from the background universe. Occurrence is then defined as external membership, i.e.,  $x\form^\mathcal{M}\mathbf{s}_\alpha \text{ iff } x\in\mathbf{s}_\alpha$; and define the stage order by
$\mathbf{s}_\alpha<^\mathcal{M}\mathbf{s}_\beta\text{ iff }\alpha<\beta.$
This completely defines the two-sorted structure
$\mathcal{M}=\langle M,\mathcal{S},A,\in,
\form^\mathcal{M},<^\mathcal{M}\rangle.$ 

Directly from the definition of availability, for each $n< \omega$, we have
$\check{\mathbf{s}}_n^\mathcal{M}
=\bigcup_{i<n}\mathbf{s}_i\cup\{a_n\}$ and for each infinite $\alpha$, $\check{\mathbf{s}}_\alpha^\mathcal{M} =\bigcup_{\beta<\alpha}\mathbf{s}_\beta.$
So it follows that $$\mathbf{s}_n
=\mathcal{P}(\check{\mathbf{s}}_n^\mathcal{M})
\cup\{a_n\}$$ for every $n < \omega$ and $\mathbf{s}_\alpha
=\mathcal{P}(\check{\mathbf{s}}_\alpha^\mathcal{M})$ for every infinite $\alpha$.

We verify that $\mathcal{M}\models\STU$. \UrDef and \Extensionality are inherited from the ambient universe, and \OrderAxiom follows from the ordering of the indices. \Staging holds because $M$ is the union of the stages. For \Priority, suppose that $y\in x \form^\mathcal{M}\mathbf{s}_\alpha$. Since $x$ is a set, $x\subseteq\check{\mathbf{s}}_\alpha^\mathcal{M}$, so $y$ is available at $\mathbf{s}_\alpha$. For \Specification, fix $\mathbf{s}_\alpha$ and suppose that every object satisfying $\varphi(z)$ is available there. Then
$$\{z\in M\mid\mathcal{M}\models\varphi(z)\}
\subseteq\check{\mathbf{s}}_\alpha^\mathcal{M},$$
so this set belongs to $\mathcal{P}(\check{\mathbf{s}}_\alpha^\mathcal{M})\subseteq\mathbf{s}_\alpha$. Hence it is formed at $\mathbf{s}_\alpha$. This proves that $\mathcal{M}\models\STU$. Since \STU proves \LTUMinus by \cref{thm:stu-proves-ltu-minus}, the object structure of $\mathcal{M}$ satisfies \LTUMinus. We may therefore use within $\mathcal{M}$ all the results of Subsection 2.3 proved from \LTUMinus. We now proceed to prove that $\mathcal{M}$ does not satisfy \Stratification.

\begin{claim}\label{claim:countermodel-x-formed}
For each $n<\omega$, put $A_n=\{a_i\mid i<n\}$. Then $\bigcup_{n<\omega}V_\omega(A_n)$ is formed at $\mathbf{s}_\omega$.
\end{claim}
\begin{proof}[Proof of the claim]
For fixed $n$, induction on $m<\omega$ gives
$$V_m(A_n)\subseteq\check{\mathbf{s}}_{n+m}^\mathcal{M}.$$
The case $m=0$ holds because every urelement in $A_n$ occurs below $\mathbf{s}_n$. If the claim holds for $m$, then every subset of $V_m(A_n)$ is formed at $\mathbf{s}_{n+m}$ and is therefore available at $\mathbf{s}_{n+m+1}$, while the urelements in $A_n$ remain available. Hence
$$V_{m+1}(A_n)=\mathcal{P}(V_m(A_n))\cup A_n
\subseteq\check{\mathbf{s}}_{n+m+1}^\mathcal{M}.$$
Thus $V_\omega(A_n)\subseteq\check{\mathbf{s}}_\omega^\mathcal{M}$ for every $n$, and therefore $\bigcup_{n<\omega}V_\omega(A_n)\subseteq\check{\mathbf{s}}_\omega^\mathcal{M}.$ Since $\mathbf{s}_\omega=\mathcal{P}(\check{\mathbf{s}}_\omega^\mathcal{M})$, it follows that $\bigcup_{n<\omega}V_\omega(A_n) \form^\mathcal{M}\mathbf{s}_\omega$.
\end{proof}

\begin{claim}\label{claim:no-full-limit-stage}
No stage of $\mathcal{M}$ forms $V_\omega(A)$.
\end{claim}
\begin{proof}[Proof of the claim]
We first show by induction on $m<\omega$ that
$$V_m(A)\form^\mathcal{M}\mathbf{s}_{\omega+m}
\quad\text{but}\quad
V_m(A)\notin\check{\mathbf{s}}_{\omega+m}^\mathcal{M}.$$
Suppose the claim holds for $m$. Every subset of $V_m(A)$ is formed at $\mathbf{s}_{\omega+m}$ and is therefore available at $\mathbf{s}_{\omega+m+1}$; all urelements in $A$ are also available there. Hence $V_{m+1}(A)$ is formed at $\mathbf{s}_{\omega+m+1}$. But if it were already available there, it would have been formed at $\mathbf{s}_{\omega+m}$; since $V_m(A)\in V_{m+1}(A)$, \Priority would then make $V_m(A)$ available at $\mathbf{s}_{\omega+m}$, contrary to the induction hypothesis.

$V_\omega(A)$ cannot be formed at any finite stage. But for any $\mathbf{s}_{\omega+m}$, the member $V_m(A)$ of $V_\omega(A)$ is not available there. Thus the members of $V_\omega(A)$ are never all available at any stage, so $V_\omega(A)$ is not formed at any stage in $\mathcal{M}$.
\end{proof}

Suppose \textit{for reductio} that $\mathcal{M}$ satisfies \Stratification. Let $x=\bigcup_{n<\omega}V_\omega(A_n)$, which is a set in $\mathcal{M}$ by \cref{claim:countermodel-x-formed}. So $x\subseteq s$ for some level $s$ in $\mathcal{M}$. Since $\ker(x)=A$, we have $\ker(s)=A$. For every $m<\omega$, $V_m(A)$ is a level with kernel $A$ in $\mathcal{M}$ and thus comparable with $s$ by \cref{lem:same-kernel-levels-comparable}. It is clear that $x \not\subseteq V_m (A)$ and so $s \not\subseteq V_m(A)$. Hence, $V_m(A)\subseteq s$ for every $m<\omega$ and so $V_\omega(A) \subseteq s.$ By \Staging, $s$ is formed at some stage $\mathbf{s}_\alpha$ so
$V_\omega(A) \subseteq s\subseteq\check{\mathbf{s}}_\alpha^\mathcal{M}$. $V_\omega(A)$ is therefore formed at $\mathbf{s}_{\alpha}$, contradicting \cref{claim:no-full-limit-stage}.\end{proof}

\subsection{Diagnosis}
Two theories are said to be \textit{set-theoretically equivalent} if they prove the same sentences in the common language of set theory (i.e., $\mathcal{L}_\Ur$ in our case). Thus we have shown that \LTU and \STU are \textit{not} set-theoretically equivalent, which is in direct contrast with Button's result that \LT and \ST are set-theoretically equivalent (\cite[Theorem 4.1]{Button2021LevelTheory1}). A puzzling situation arises. On the one hand, \STU is a literal translation of the Basic Iterative Story. On the other, \LTU is also a faithful formalization since stages are naturally understood as levels. Yet \LTU proves more set-theoretic truths than \STU. Note that one should not blame this on not taking \UrStage as an axiom of \STU: $\mathcal{M}$ in Theorem \ref{thm:stu-bizarre-model} satisfies \UrStage and hence \UrSet because all urelements are available on stage $\mathbf{s}_\omega$.

Two diagnoses are available. The first is that \LTU smuggles something beyond the Basic Iterative Story into its formalization. Our fixed-kernel definition of level requires that the urelements available throughout any one iterative process be exactly those in its initial level. By contrast, in \cref{Section:Intro} we motivated a picture in which new urelements may become available at later stages, a possibility permitted by both the Basic Iterative Story and \STU. It is not obvious, however, that this difference should account for the failure of set-theoretic equivalence. \LTU simply places the additional urelements in different initial levels and lets them generate distinct hierarchies, and this difference in the location of urelements should not make \LTU prove more set-theoretic statements. Whether an alternative definition of level would yield a weaker level theory set-theoretically equivalent to \STU is not known, but we shall not pursue it here (for the reason mentioned below).

Instead, we propose a different diagnosis: the Basic Iterative Story with urelements, unlike in the pure-set context, has certain arbitrariness regarding how sets are formed, which is why it admits different formalizations. Another look at the model $\mathcal{M}$ in \cref{thm:stu-bizarre-model} reveals the arbitrariness: different sets of urelements in $\mathcal{M}$ generate hierarchies of different height. For instance, $\{a_0\}$ generates hierarchies of infinite height such as $V_\omega (\{a_0\})$ while $A$ fails to generate any infinite hierarchy. The stronger theory \LTU cannot exclude this pathological situation either. Consider the model
$$\mathcal{N} =  V_{\omega + \omega} (\{a_0\}) \cup V_\omega (\{a_1\}).$$
It follows from \cref{thm:unions-of-hierarchies} that $\mathcal{N}$ is a model of \LTU. But inside $\mathcal{N}$, $\{a_1\}$ generates no infinite hierarchy while $\{a_0\}$ does. The iterative process built from $a_1$ stopped in a very arbitrary way given that another iterative process has gone further. And as we will discuss in \cref{section:quasi-categoricity}, this kind of \textit{uneven model} will result in a severe failure of quasi-categoricity. So by Theorem \ref{thm:ltu-interprets-stu}, we know that even if there is a version of level theory that is set-theoretically equivalent to \STU, it will also have uneven models and thus face an equally bad failure of quasi-categoricity. Thus, some fundamental principles seem to be missing in the Basic Iterative Story. We now turn to extensions of this story that may lead to a robust conception of set.

\section{Extending the Basic Story}\label{sec:extendingBIS}
In this section, we will consider natural extensions of \STU and \LTU by adding further axioms  and discuss the relationship among these axioms as well as the relationship between correspondingly extended theories. The additional axioms we shall consider are standard and well-known in the literature, but, as we shall see, their relationship and individual strength are completely obscured by the traditional treatment of assuming \UrSet.

\subsection{Extending \STU}
One obvious question left unanswered by the basic iterative conception is, of course, how many stages there are. And as we observed before, one arbitrariness with the Basic Iterative Story is that some iterative process might stop at a place where it should not have stopped. So to extend the Basic Iterative Story, a rule of thumb is to find natural principles asserting that there are as many stages as possible so that every iterative process goes as far as possible. Let us begin with a rather mild principle that the stages are \textit{upward directed}.
\begin{itemize}
\item [] (\StageDirected) \(\exists\mathbf{t}(\mathbf{r}<\mathbf{t}\land\mathbf{s}<\mathbf{t})\).
\end{itemize}
In the pure setting (or assuming \UrSet), one could consider the weaker version that every stage is below some stage (which Button calls Endless). But in our more general setting, since there can be incomparable stages, \StageDirected is a more natural further principle to propose.

A more powerful principle says that stages are unbounded in sets.
\begin{itemize}

\item [] (\StageUnbounded) \(\forall x\in w\exists!y\,\varphi(x,y)\rightarrow\exists\mathbf{s}\forall x\in w\exists y(\varphi(x,y)\land y\avail\mathbf{s})\).
\end{itemize}
\noindent That is, if there is a map from a set to some objects, then there is a stage at which these objects are available.  It will be an immediate consequence of \StageUnbounded that the \Replacement scheme holds (Proposition \ref{prop:stage-axiom-consequences}).
\begin{itemize}
\item [] (\Replacement)\footnote{The variables $w$ and $v$ are intended to range over sets. But the sethood predicates can be omitted given \UrDef.} \(\forall x\in w\exists!y\,\varphi(x,y)\rightarrow\exists v\forall y(y\in v\leftrightarrow\exists x\in w\,\varphi(x,y))\).
\end{itemize}
\noindent Interestingly, Boolos in \cite[228]{Boolos1971IterativeConception} considers the following unbounded principle to motivate \Replacement and remarks that it is a \textit{further thought} in addition to the iterative conception.
\begin{itemize}
\item [] (\StageUnboundedPlus) \(\forall x\in w\exists\mathbf{r}\,\varphi(x,\mathbf{r})\rightarrow\exists\mathbf{s}\forall x\in w\exists\mathbf{r}(\varphi(x,\mathbf{r})\land\mathbf{r}<\mathbf{s})\).
\end{itemize}
\noindent That is, if every member of a set is related to some stage, then there is a stage such that every member of a set  is related to some stage below this stage.

Note that \StageUnboundedPlus does not require the relation to be a map. Even if every member of the set has ``class-many'' stages as witnesses, there is still one stage, according to \UnboundedPlus, above enough related stages. \StageUnboundedPlus yields the \Collection scheme (Proposition \ref{prop:stage-axiom-consequences}).
\begin{itemize}

\item [] (\Collection) \(\forall x\in w\exists y\,\varphi(x,y)\rightarrow\exists v\forall x\in w\exists y\in v\,\varphi(x,y)\).
\end{itemize}
\noindent It is a textbook result that \Collection is equivalent to \Replacement over the remaining axiom of \ZFU + \UrSet.

Finally, let us consider the idea of \textit{reflection}. Set-theoretic reflection principles assert that the set-theoretic universe is \textit{indescribable}, i.e., unable to be characterized by any statement. However, there are two ways in which this indescribability idea can be spelled out. One is \textit{partial reflection}.
\begin{itemize}
\item [] \PartialReflection Any true statement is true in some initial fragment of the universe.
\end{itemize}
The other is the stronger notion of \textit{complete reflection}.
\begin{itemize}
\item [] \CompleteReflection For any statement $\varphi$, there is an arbitrarily large fragment such that $\varphi$ is absolute between the fragment and the universe.
\end{itemize}
These two forms of reflection are rarely distinguished in the literature. This is mainly because no such distinction can be made in \ZF due to the following classical result.
\begin{itemize}
\item [] (\LevyMontagueReflection) \(\forall\alpha\exists\beta>\alpha\forall\vec{x}\in V_\beta(\varphi(\vec{x})\leftrightarrow\varphi^{V_\beta}(\vec{x}))\).
\end{itemize}
\noindent \LevyMontagueReflection is a form of complete reflection: it asserts that for every formula, there is an arbitrarily tall hierarchy such that $\varphi$ is absolute between the hierarchy and the universe, which immediately implies the following form of partial reflection.
\begin{itemize}
\item []  $\forall \vec{x} (\varphi(\vec{x}) \to \exists \alpha (\vec{x} \in V_\alpha \land \varphi^{V_\alpha}(\vec{x})))$
\end{itemize}
\noindent Furthermore, the proof of \LevyMontagueReflection easily generalizes to \ZFU + \UrSet, where $V_\alpha$ is replaced with $V_\alpha (A)$ for the set $A$ of all urelements. In fact, over the remaining axioms of \ZFU + \UrSet, \Replacement, \Collection and the \LevyMontagueReflection are equivalent. Thus, given the tradition of excluding urelements or assuming \UrSet, it is common to view the two versions of unboundedness principles, together with \CompleteReflection, as the same way of extending the basic iterative conception of set; and consequently, \PartialReflection comes for free.\footnote{See Paseau \cite{Paseau2007BoolosJustification} for treating a form of partial reflection as part of the iterative conception.}  As we will see, this equivalence breaks down once we drop \UrSet.

Now we formulate reflection principles in stage theory by making the notion of initial fragment and ``true in a fragment'' precise in $\mathcal{L}_\textbf{stage}$. There is a natural way to proceed: an initial fragment is simply a slice of a stage (Definition \ref{def:stu-slice}) together with the stages below. So ``$\varphi$ is true in \textbf{s}'' will be a form of \textit{stage relativization} defined as follows.
\begin{definition}\label{def:stage-relativization}
Let $\varphi$ be a formula in $\mathcal{L}_\textbf{stage}$ and let $\mathbf{s}$ be a stage variable. The \textit{stage relativization} $\varphi^{\mathbf{s}}$ is defined recursively as follows:
\begin{itemize}
\item [] if $\varphi$ is atomic, then $\varphi^{\mathbf{s}}=\varphi$;

\item [] $(\neg\psi)^{\mathbf{s}}=\neg\psi^{\mathbf{s}}$ and $(\psi\land\theta)^{\mathbf{s}}=\psi^{\mathbf{s}}\land\theta^{\mathbf{s}}$;

\item [] $(\exists x\,\psi)^{\mathbf{s}}=\exists x(x\avail\mathbf{s}\land\psi^{\mathbf{s}})$;

\item [] $(\exists\mathbf{r}\,\psi)^{\mathbf{s}}=\exists\mathbf{r}(\mathbf{r}<\mathbf{s}\land\psi^{\mathbf{s}})$.
\end{itemize}
\end{definition}
\noindent Accordingly, we can formulate the partial reflection principle \StageRPMinus and the complete reflection principle \StageRP as follows.
\begin{itemize}
\item [] (\StageRPMinus) \(\varphi(\vec{x},\vec{\mathbf{r}})\rightarrow\exists\mathbf{s}(\vec{x}\avail\mathbf{s}\land\vec{\mathbf{r}}<\mathbf{s}\land\varphi^{\mathbf{s}}(\vec{x},\vec{\mathbf{r}}))\).\\

\item [] {\small (\StageRP) \(\forall\vec{\mathbf{t}}\exists\mathbf{s}(\vec{\mathbf{t}}<\mathbf{s}\land\forall\vec{x}\avail\mathbf{s}\,\forall\vec{\mathbf{r}}<\mathbf{s}(\varphi(\vec{x},\vec{\mathbf{r}})\leftrightarrow\varphi^{\mathbf{s}}(\vec{x},\vec{\mathbf{r}})))\).}
\end{itemize}
\noindent We now verify some immediate set-theoretic consequences of these axioms and then clarify the relationship among these axioms.
\begin{prop}[\STU]\label{prop:stage-axiom-consequences}\
\begin{enumerate}
\item \StageDirected implies \Powerset and \Pairing.

\item \StageUnbounded implies \Replacement.

\item \StageUnboundedPlus implies \Collection.

\end{enumerate}
\end{prop}

\begin{proof}
(1) For any $x$ and $y$, \Staging and \StageDirected give a stage $\mathbf{t}$ above two stages at which they occur. Thus $x,y \prec\mathbf{t}$, and \Specification at $\mathbf{t}$ forms $\{x,y\}$. For \Powerset, let $x$ be a set formed at some stage $\mathbf{r}$; \StageDirected, applied to $\mathbf{r}$ twice, gives a stage $\mathbf{t}$ above $\mathbf{r}$. If $y\subseteq x$, then $y \subseteq \check{\mathbf{r}}$ so $y\form\mathbf{r}$. Hence every subset of $x$ is available at $\mathbf{t}$, and \Specification at $\mathbf{t}$ forms $\Pow(x)$.

(2) Suppose $\forall x\in w\exists!y\,\varphi(x,y)$. \StageUnbounded gives a stage $\mathbf{s}$ at which every value is available. \Specification at $\mathbf{s}$, applied to $\exists x\in w\,\varphi(x,z)$, forms $v=\{z\mid\exists x\in w\,\varphi(x,z)\},$ as desired.

(3) Suppose $\forall x\in w\exists y\,\varphi(x,y)$. By \Staging, we have $\forall x\in w\exists\mathbf{r}\exists y\bigl(\varphi(x,y)\land y\form\mathbf{r}\bigr).$
By \StageUnboundedPlus, there is a stage $\mathbf{s}$ such that for every $x\in w$ there is some $\mathbf{r}<\mathbf{s}$ with $\exists y (\varphi (x,y) \land y \form\mathbf{r})$. \Specification at $\mathbf{s}$, applied to $$z \avail\mathbf{s}\land\exists x\in w\,\varphi(x,z),$$ forms a set containing a witness for every $x\in w$. Hence \Collection holds.
\end{proof}

\begin{prop}\label{prop:stage-axiom-relations}
Over \STU, the following implication diagram holds.
\begin{center}
\begin{tikzpicture}[scale=0.8,transform shape]
\node (rp) at (0,3) {$\StageRP$};
\node (rpminus) at (-2,1.5) {$\StageRPMinus$};
\node (unboundedplus) at (2,1.5) {$\StageUnboundedPlus$};
\node (directed) at (-2,0) {$\StageDirected$};
\node (unbounded) at (2,0) {$\StageUnbounded$};
\draw[->] (rp) -- (rpminus);
\draw[->] (rp) -- (unboundedplus);
\draw[->] (rpminus) -- (directed);
\draw[->] (unboundedplus) -- (unbounded);
\end{tikzpicture}

\refstepcounter{implicationdiagram}
\label{diagram:stu-implications}
Implication Diagram \theimplicationdiagram
\end{center}
\end{prop}

\begin{proof}
(\StageUnboundedPlus $\to$ \StageUnbounded) Suppose $\forall x\in w\exists!y\,\varphi(x,y).$
By \Staging, for every $x\in w$ there is a stage at which its $\varphi$-image occurs. By \StageUnboundedPlus, there is a stage $\mathbf{s}$ such that for every $x\in w$, its $\varphi$-image occurs at some $\mathbf{r}<\mathbf{s}$ and is therefore available at $\mathbf{s}$. This is \StageUnbounded.

(\StageRPMinus $\to$ \StageDirected) Given stages $\mathbf{r}$ and $\mathbf{t}$, apply \StageRPMinus to the true formula $\mathbf{r}=\mathbf{r}\land\mathbf{t}=\mathbf{t}.$ The resulting stage $\mathbf{s}$ is such that $\mathbf{r},\mathbf{t}<\mathbf{s}$.

(\StageRP $\to$ \StageRPMinus) Suppose $\varphi(\vec{x},\vec{\mathbf r})$ holds. There are stages $\vec{\mathbf t}$ at which the objects $\vec{x}$ occur. Apply \StageRP by reflecting $\varphi$ with a stage $\mathbf{s}$ above $\vec{\mathbf t}$ and $\vec{\mathbf r}$. Then $\vec{x}\avail\mathbf{s}$ and $\vec{\mathbf r}<\mathbf{s}$; so $\varphi^{\mathbf{s}}(\vec{x},\vec{\mathbf r})$.

(\StageRP $\to$ \StageUnboundedPlus) We first review a trick of simultaneously reflecting two formulas $\varphi$ and $\psi$ with the same stage. By \Specification and \StageDirected, $\emptyset$ and $\{\emptyset\}$ both exist and occur at some stage $\mathbf{t}$. Define $\theta(u,\vec{x},\vec{\mathbf r})$ as
\begin{equation*}
(u=\emptyset\land\varphi(\vec{x},\vec{\mathbf r}))\lor
(u=\{\emptyset\} \land\psi(\vec{x},\vec{\mathbf r})).
\end{equation*}
By \StageRP, there is a stage $\textbf{s} > \textbf{t}$ that reflects  $\theta$. Since both $\emptyset$ and $\{\emptyset\}$ are available at $\mathbf{s}$, it follows that both $\varphi$ and $\psi$ are absolute between $\mathbf{s}$ and the universe.

Now suppose $\forall x\in w\exists\mathbf{r}\,\varphi(x,\mathbf{r})$ and let $\chi(w)$ abbreviate this assertion. Let \textbf{t} be a stage at which $w$ occurs. By simultaneous reflection, there is a stage $\mathbf{s} > \textbf{t}$ which reflects both $\chi$ and $\varphi$. $w$ and every member of $w$ is therefore available at $\mathbf{s}$. Since $\chi (w)$ holds, we have $\chi(w)^{\mathbf{s}}$ so for every $x\in w$, there is a stage $\mathbf{r}<\mathbf{s}$ such that $\varphi^{\mathbf{s}}(x,\mathbf{r})$ and hence $\varphi(x,\mathbf{r})$. Therefore, \StageUnboundedPlus holds.
\end{proof}

\noindent A natural question is whether \cref{diagram:stu-implications} is complete: is there any implication that we have not proved? This leads to philosophical questions, e.g., does \CompleteReflection express a conception of set that is stronger than the other conceptions expressed by other axioms? We will answer this question through studying the counterpart of these axioms in level theory, to which we now turn.

\subsection{Extending \LTU}
By replacing stage with level in the first three axioms, we have their counterparts in level theory.
\begin{itemize}
\item [] (\LevelDirected) \(\Lev(r)\land\Lev(s)\rightarrow \exists t(\Lev(t)\land r\in t\land s\in t)\).\\

\item [] {\small (\LevelUnbounded) \(\forall x\in w\exists!y\,\varphi(x,y)\rightarrow \exists s(\Lev(s)\land \forall x\in w\exists y\in s\,\varphi(x,y))\).}\\

\item [] {\small (\LevelUnboundedPlus) \(\forall x\in w\exists r(\Lev(r)\land\varphi(x,r))\rightarrow \exists s(\Lev(s)\land\forall x\in w\exists r\in s(\Lev(r)\land\varphi(x,r)))\).}
\end{itemize}

\noindent Reflection needs more care. It is tempting to treat initial fragments of the universe as levels, which leads to the following partial reflection principle.
\begin{itemize}
\item [] (\LevelRPMinus) \(\varphi(\vec{x})\rightarrow \exists s(\Lev(s)\land\vec{x}\in s\land\varphi^s(\vec{x}))\),
\end{itemize}
where $\varphi^s$ is the standard $\in$-relativization. But by \cref{lem:history-kernel} for every level $s$ that has a set member, $\ker(s) \in s$; let $A = \ker(s)$ and we have
\begin{itemize}
\item []  $A \in s \land \forall a \in s(\Ur(a) \to a \in A)$
\end{itemize}
\noindent which implies (\UrSet)$^s$. Thus, \LevelRPMinus immediately implies \UrSet since $\neg\UrSet \land \exists x\,\Set(x)$ cannot be reflected by any level. This is no surprise: every level is fundamentally different from the universe whenever the urelements do not form a set. Instead, initial fragments of the universe should be simply understood as sets that are closed under membership and set-formation, i.e., supertransitive sets. Accordingly, \PartialReflection and \CompleteReflection are better formulated as follows.
\begin{itemize}
\item [] (\RPSim) \(\varphi(\vec{x})\rightarrow \exists t( t\text{ is supertransitive}\land\vec{x}\in t\land\varphi^t(\vec{x}))\).\\

\item [] (\RP) \(\forall \vec{w} \exists t( \vec{w} \in t\land t\text{ is supertransitive}\land\forall\vec{x}\in t(\varphi(\vec{x})\leftrightarrow\varphi^t(\vec{x})))\).\footnote{We note that \RP has many different equivalent formulations over \LTU; for example, the set $t$ in \RP needs only to be transitive. There is also a weaker version of \RPSim, denoted by \RPMinus, which requires the reflecting set $t$ to be merely transitive. Since \RPMinus plays no role in what follows, we omit it from the discussion.}
\end{itemize}

\begin{prop}\label{prop:LTUImplication}
The following implication diagram holds over \LTU.

\begin{center}
\begin{tikzpicture}[scale=0.8,transform shape]
\node (rp) at (0,3) {$\RP$};
\node (rpsim) at (-2,1.5) {$\RPSim$};
\node (unboundedplus) at (2,1.5) {$\LevelUnboundedPlus$};
\node (collection) at (5.5,1.5) {$\Collection$};
\node (directed) at (-2,0) {$\LevelDirected$};
\node (unbounded) at (2,0) {$\LevelUnbounded$};
\node (replacement) at (5.5,0) {$\Replacement$};
\draw[->] (rp) -- (rpsim);
\draw[->] (rp) -- (unboundedplus);
\draw[->] (rpsim) -- (directed);
\draw[->] (unboundedplus) -- (unbounded);
\draw[<->] (unboundedplus) -- (collection);
\draw[<->] (unbounded) -- (replacement);
\end{tikzpicture}

\refstepcounter{implicationdiagram}
\label{diagram:ltu-implications}
Implication Diagram \theimplicationdiagram
\end{center}

\end{prop}

\begin{proof}
The arguments for \RP $\to$ \RPSim, \RPSim $\to$ \LevelDirected, \LevelUnboundedPlus $\to$ \LevelUnbounded, \LevelUnbounded $\leftrightarrow$ \Replacement and \LevelUnboundedPlus $\leftrightarrow$ \Collection are completely analogous to \cref{prop:stage-axiom-relations}, so we omit the proofs.

The argument for \RP $\to$  \LevelUnboundedPlus is also similar. As before, one first observes that \RP gives finite simultaneous reflection. Now suppose that $\forall x\in w\exists r(\Lev(r)\land\varphi(x,r)),$ and let $\chi(w)$ abbreviate this assertion. By simultaneous reflection, there is a transitive set $t$ containing $w$ which reflects both $\chi$ and the formula $\Lev(r)\land\varphi(x,r)$. Since $\chi (w)$ holds, we have $\chi(w)^t$ so for every $x\in w$ there is some $r\in t$ such that $(\Lev(r)\land\varphi(x,r))^t$ and hence $\Lev(r)\land\varphi(x,r)$. Let $s$ be a level with $t\subseteq s$. Then for every $x\in w$, there is a level $r\in s$ such that $\varphi(x,r)$, so $s$ witnesses \LevelUnboundedPlus. \end{proof}

\begin{theorem}\label{thm:ltu-implication-diagram-complete}
\Cref{diagram:ltu-implications} is complete over \LTU.
\end{theorem}
\begin{proof}
The model whose only object is the empty set satisfies \LTU + \Collection but not \LevelDirected. L\'evy and Vaught \cite[Corollary 8]{LevyVaught1961PartialReflection} show that \RPSim does not imply \Replacement over Zermelo set theory. Their model satisfies $\forall x \exists \alpha x \in V_\alpha$, which implies \Stratification, so it is a model of \LTU. Finally, the main results in Glazer and Yao \cite[Section 3]{GlazerYao2026ReflectionZFU} show that the following diagram is complete over \ZFU.
\begin{center}
\begin{tikzpicture}[scale=0.8,transform shape]
\node (rp) at (0,3) {$\RP$};
\node (rpsim) at (-2,1.5) {$\RPSim$};
\node (collection) at (2,1.5) {$\Collection$};
\draw[->] (rp) -- (rpsim);
\draw[->] (rp) -- (collection);
\end{tikzpicture}
\end{center}
As \ZFU proves \LTU + \LevelDirected + \Replacement, it follows that \cref{diagram:ltu-implications} is complete over \LTU.\end{proof}

\section{Set-Theoretic Equivalence Regained}\label{sec:Equivalence}

We have seen that \LTU and \STU are not set-theoretically equivalent: \STU does not prove the key axiom \Stratification (\cref{thm:stu-bizarre-model}). We argued that this non-equivalence result is due to the non-robustness of the Basic Iterative Story. Now we consider whether adding the further principles introduced in the previous section can lead to set-theoretic equivalence. When corresponding stage- and level-theoretic extensions have the same
set-theoretic consequences, this provides evidence that the added principle resolves an
ambiguity between the two formalizations of the iterative process and therefore makes the conception of set more robust. We start by observing that directedness is no such principle.

\begin{theorem}\label{thm:stage-directed-not-ltu}
\STU + \StageDirected does not prove either \Stratification or \LevelDirected.
\end{theorem}

\begin{proof}
Consider the model $\mathcal{M}$ in \cref{thm:stu-bizarre-model}. We already know that $\mathcal{M}\models\STU$ and that \Stratification fails in $\mathcal{M}$. The model clearly satisfies \StageDirected since the stage order has order type $\omega+\omega$.

We show that \LevelDirected fails. Recall that $A_1=\{a_0\}$ and $A = \{a_n \mid n <\omega\}$. Consider $r=V_1(A)\quad\text{and}$ $s=V_\omega(A_1)$, which are both levels in $\mathcal{M}$. Suppose  \textit{for reductio} that there is a level $t$ in $\mathcal{M}$ with  $r,s\in t$. Then $\ker(t)=A$ and so $t$ is comparable with $V_m(A)$ for every $m < \omega$ by \cref{lem:same-kernel-levels-comparable}. $t=V_m(A)$ and $t\in V_m(A)$ are impossible, since $V_\omega(A_1)\in t$ but $V_\omega(A_1)\notin V_m(A)$. Therefore, $V_m(A)\in t$ for every $m<\omega$ so we have $V_\omega(A) \subseteq t$. $\mathcal{M}$ would therefore form $V_\omega(A)$ at some stage, contrary to \cref{claim:no-full-limit-stage}.\end{proof}

\subsection{Equivalence via Unboundedness}
We observed after \cref{thm:stu-bizarre-model} that \LTU seems too weak because the hierarchies generated by different initial levels might be uneven. This is against a very mild maximality conception of set, i.e., one iterative process should not stop if there is another one that has gone further. It turns out that such a principle, what we call \LevelExtension, is the key to achieving set-theoretic equivalence.

\begin{itemize}
\item [] (\LevelExtension) For every level $s$ and every set $A$ of urelements such that $\ker(s)\subseteq A$, there is a level $t$ such that $s\subseteq t$ and $\ker(t)=A$.
\end{itemize}

\noindent That is, every level can be extended to a larger level with more urelements.
\begin{lemma}[\LTUMinus]\label{lem:level-extension}
\Replacement implies \LevelExtension.
\end{lemma}
\begin{proof}
Work in \LTUMinus + \Replacement. Fix a set $A$ of urelements, and suppose that some level whose kernel is included in $A$ has no extension with kernel $A$. By Potent Induction, fix such a level $s$ that is $\in$-minimal. Let
\begin{equation*}
S=\{r\in s\mid \textup{Lev}(r)\land \ker(r)=\ker(s)\}.
\end{equation*}
\noindent For each $r\in S$, the minimality of $s$ implies that there is some level $l$ such that $r\subseteq l$ and $\ker(l)=A$; and by \cref{lem:same-kernel-levels-comparable}, there is a unique $\in$-least such level $l_r$. By \Replacement, $\bar h=\{l_r\mid r\in S\}$ is a set. Define
\begin{equation*}
h=\{l \mid \textup{Lev}(l)\land \ker(l)=A\land (l \in\bar h\lor \exists w\in\bar h(l \in w))\} \cup A.
\end{equation*}
\noindent The set $h$ exists: if $\bar h$ is empty, then $h=A$; otherwise, $h$ is a subset of a supertransitive set containing $\bar h$ and hence $A$, which exists by \StratificationMinus. The urelement members of $h$ are exactly the members of $A$, and every set in $h$ is a level with kernel $A$. If $q\in l\in h$ and $q$ is a level with kernel $A$, the transitivity of levels gives $q\in h$. Hence \cref{lem:fixed-kernel-history} shows that $h$ is a history and $\ker(h)=A$.

Let $t=\pot(h)$. Then $t$ is a level and $\ker(t)=A$ by \cref{lem:pot-kernel}. Let $x\in s$. If $x$ is a urelement, then $x\in\ker(s)\subseteq A$, so $x\in t$. If $x$ is a set, then \cref{lem:level-decomposition} gives an $r\in S$ such that $x\subseteq r$. Since $r\subseteq l_r\in h$, we have $x\in\pot(h)=t$. Therefore $s\subseteq t$, contradicting the choice of $s$. This proves \LevelExtension.
\end{proof}

\begin{lemma}[\STU]\label{lem:stu-stage-unbounded-slices-bounded}
\StageUnbounded implies that every slice is a subset of a level.
\end{lemma}

\begin{proof}
Work in \STU +\StageUnbounded.  Note that we can use all the consequences of \LTUMinus + \Replacement by Theorem \ref{thm:stu-proves-ltu-minus} and Proposition \ref{prop:stage-axiom-consequences}, including \LevelExtension by Lemma \ref{lem:level-extension}. Suppose some slice is not contained in any level. Every slice is potent by \cref{lem:stu-slice-properties} so by Potent Induction, fix an $\in$-minimal slice $x$ which is not contained in any level. Let $A = \ker(x)$ and $H_x=\{q\in x\mid \Slice(q)\}.$
Let $q\in H_x$. By the minimality of $x$, there is a level $m$ such that $q\subseteq m$. By the \hyperref[lem:level-shrink]{Level-Shrink Lemma}, 
$$m^{\ker(q)}=\{z\in m\mid \Set(z)\land\ker(z)\subseteq\ker(q)\}\cup\ker(q)$$
\noindent is a level with kernel $\ker(q)$, and $q\subseteq m^{\ker(q)}$. Since $\ker(q) \subseteq A$, by \LevelExtension, there is a level $l$ such that $m^{\ker(q)} \subseteq l$ and $\ker(l)=A$. Thus, for each $q\in H$, there is a level with kernel $A$ which contains $q$. For each $q\in H_x$, let $l_q$ be the unique $\in$-least level $l$ such that $q\subseteq l$ and $\ker(l)=A$, which exists by \cref{lem:same-kernel-levels-comparable}. By \Replacement, $\bar h=\{l_q\mid q\in H_x\}$ is a set.

Now extend $\bar h$ as in the previous lemma by letting
\begin{equation*}
h=\{v\mid \textup{Lev}(v)\land \ker(v)=A\land (v\in\bar h\lor \exists w\in\bar h(v\in w))\}\cup A.
\end{equation*}

\noindent The same argument shows that \cref{lem:fixed-kernel-history} applies, so $h$ is a history and $\ker(h)=A$. Let $t=\pot(h)$.  If $a$ is a urelement in $x$, then $a \in A=\ker(h)$ so $a\in t$. If $y \in x$ is a set, then by \cref{lem:stu-slice-decomposition}, there is $q\in H_x$ such that $y\subseteq q$. Since $q\subseteq l_q\in h$, we have $y\in\pot(h)=t$. Thus $x\subseteq t$, which is a level, contradicting the definition of $x$.
\end{proof}

\begin{theorem}\label{thm:stage-level-unbounded-equivalence}
The following pairs of theories are set-theoretically equivalent.
\begin{enumerate}
\item  \STU+\StageUnbounded and \LTU + \LevelUnbounded.

\item \STU+\StageUnboundedPlus and \LTU+\LevelUnboundedPlus.
\end{enumerate}
\end{theorem}

\begin{proof}
We first show that \STU+\StageUnbounded proves \LTU + \LevelUnbounded. Assume the former theory. By \cref{prop:stage-axiom-consequences}(2), \Replacement holds, and \cref{thm:stu-proves-ltu-minus} gives \UrDef, \Extensionality, and \Separation. For \Stratification, fix an object $x$ and a stage $\mathbf{s}$ such that $x\form\mathbf{s}$. If $x$ is a set, \cref{lem:stu-slice-properties}(2) gives $x\subseteq\check{\mathbf{s}}$. If $x$ is a urelement, then $x\in\check{\mathbf{s}}$ by the definition of availability. By \cref{lem:stu-stage-unbounded-slices-bounded}, there is a level $l$ such that $\check{\mathbf{s}}\subseteq l$. Hence $x\subseteq l$ in the set case and $x\in l$ in the urelement case. Thus \LTU holds. By \cref{diagram:ltu-implications}, \Replacement now gives \LevelUnbounded.

If \StageUnboundedPlus is assumed instead, then \StageUnbounded follows from \cref{prop:stage-axiom-relations}, and the preceding argument gives \LTU. Moreover, \Collection follows from \cref{prop:stage-axiom-consequences}(3). Hence \LevelUnboundedPlus follows from \cref{diagram:ltu-implications}, so \STU+\StageUnboundedPlus proves \LTU+\LevelUnboundedPlus.

Conversely, work first in \LTU+\LevelUnbounded and use the canonical translation in \cref{def:canonical-translation}. By \cref{thm:ltu-interprets-stu}, every axiom of \STU holds under this translation. It remains to verify the translated instances of \StageUnbounded. Let $\varphi(x,y)$ be any $\mathcal{L}_\textbf{stage}$-formula. Since
$$\bigl(y\avail\mathbf{s}\bigr)^*\longleftrightarrow y\in s$$
by the proof of \cref{thm:ltu-interprets-stu}, the translation of the corresponding instance of \StageUnbounded is
$$\forall x\in w\exists!y\,\varphi^*(x,y)\rightarrow\exists s\bigl(\Lev(s)\land\forall x\in w\exists y\in s\,\varphi^*(x,y)\bigr).$$
This is precisely the instance of \LevelUnbounded obtained by substituting $\varphi^*$ for $\varphi$. Hence \LTU+\LevelUnbounded interprets \STU+\StageUnbounded.

Now work in \LTU+\LevelUnboundedPlus. Let $\varphi(x,\mathbf{r})$ be any $\mathcal{L}_\textbf{stage}$-formula and assume$\forall x\in w\exists r\bigl(\Lev(r)\land\varphi^*(x,r)\bigr).$ Applying \LevelUnboundedPlus to $\varphi^*$ gives
$$\exists s\bigl(\Lev(s)\land\forall x\in w\exists r\in s(\Lev(r)\land\varphi^*(x,r))\bigr),$$
which is the translation of $\exists\mathbf{s}\forall x\in w\exists\mathbf{r}\bigl(\varphi(x,\mathbf{r})\land\mathbf{r}<\mathbf{s}\bigr)$. Thus \LTU+\LevelUnboundedPlus interprets \STU+\StageUnboundedPlus. Since the interpretation leaves $\mathcal{L}_\Ur$ unchanged, the two pairs of theories are set-theoretically equivalent.
\end{proof}

\noindent \STU proves that every level is contained in some slice. Together with the preceding lemma, this shows that, in \STU + \StageUnbounded, levels and slices are \textit{mutually cofinal}: every level is contained in a slice, and every slice is contained in a level.\footnote{However, \STU + \StageUnbounded cannot show that slices \textit{are exactly} the levels as in Button \cite[Lemma 4.7]{Button2021LevelTheory1}.  Intuitively, this is because we can interpret stages freely inside a model of urelement set theory. To illustrate, consider a model of \ZFU with the urelements forming a set $A=\{a_n\mid n<\omega\}$. If we interpret stages as all and only the $V_\alpha(A)$ hierarchies and interpret availability as membership, then $V_\alpha(B)$ with $B \subsetneq A$ is a level but not a slice since every slice has kernel $A$. Conversely, let $A_n=\{a_i\mid i<n\}$, and we can treat $\bigcup_{n<\omega}V_\omega(A_n)$ as an additional stage. $\bigcup_{n<\omega}V_\omega(A_n)$, as a slice, is not itself a level by \cref{lem:history-kernel}  because it has kernel $A$ but does not have $A$ as a member. Both stage interpretations satisfy \StageUnbounded since the model has \Replacement.}

\subsection{Equivalence via Reflection}
Partial reflection can also achieve set-theoretic equivalence. We first record the urelemental analogue of an absoluteness result proved by Button and Walsh \cite[Proposition 8.29]{ButtonWalsh2018PhilosophyModelTheory}.

\begin{lemma}[\LTUMinus]\label{lem:level-absoluteness-supertransitive}
Let $t$ be a supertransitive set such that (\StratificationMinus)$^t$. Then for any $r\in t$, $$\Lev(r)^t \leftrightarrow \Lev(r).$$
\end{lemma}

\begin{proof}
First, transitivity, potency, and supertransitivity are absolute for $t$. In particular, if $t$ regards $p\in t$ as potent and $x\subseteq y\in p$, then $x\in t$ by the potency of $t$, and hence $x\in p$ by the potency of $p$ in $t$.

We next show that $t$ correctly computes kernels and potentiations. Let $q\in t$ be a set and $A = \ker(q)$. Since $(\StratificationMinus)^t$, we fix some supertransitive $p\in t$ such that $q\subseteq p$. So $A\subseteq p$, and hence $A\in t$ by the potency of $t$. Moreover, $t$ also regards $A$ as the kernel of $q$. Otherwise, there would be some $a$ such that $(a \in \ker(q))^t$ but $a\notin A$. Then there is an external transitive set $y$ such that $q\subseteq y$ and $a\notin y$. Since $p\cap y\subseteq p\in t$, $p\cap y\in t$. But $p\cap y$ is transitive, contains $q$, and omits $a$, contradicting the judgment of $t$. Hence $t$ correctly computes $\ker(q)$.

Let $P=\pot(q)$. If $x\in P$ is a set, then $x\subseteq y\in q$ for some $y$, so $x\in p$ by the potency of $p$; and if $x\in P$ is a urelement, then $x\in A\subseteq p$. Thus $P\subseteq p$, and so $P\in t$ by the potency of $t$. It then follows that $t$ regards $P$ as the potentiation of $q$ and thus correctly computes $\pot(q)$. It follows that historyhood is absolute for every $h\in t$ since the terms occurring in the definition of a history are absolute.

Now suppose that $r\in t$ is an actual level, witnessed by a history $h$ such that $r=\pot(h)$. Since $h\subseteq r\in t$, $h\in t$. By the absoluteness of historyhood and potentiation, $t$ thinks that $h$ is a history with $r=\pot(h)$, so $\Lev(r)^t$. Conversely, if $t$ regards $r$ as a level, let $h\in t$ witness this. The same absoluteness argument shows that $h$ is an actual history and $r=\pot(h)$.\end{proof}

\begin{prop}[\STU]\label{prop:stage-rp-minus-rpsim}
\StageRPMinus implies \RPSim.
\end{prop}

\begin{proof}
Suppose that $\varphi(\vec{x})$ holds, where $\varphi$ is an $\mathcal{L}_\Ur$-formula. \StageRPMinus gives a stage $\mathbf{s}$ such that $\vec{x}\avail\mathbf{s}$ and $\varphi^{\mathbf{s}}(\vec{x})$. Let $t=\check{\mathbf{s}}$. By \cref{lem:stu-slice-properties}, $t$ is supertransitive and $\vec{x}\in t$. Moreover, object quantification in $\varphi^{\mathbf{s}}$ is exactly quantification over $t$. Hence, $\varphi^t(\vec{x})$. \end{proof}

\begin{lemma}[\LTUMinus]\label{lem:rpsim-level-extension}
\RPSim implies \LevelExtension.
\end{lemma}

\begin{proof}
The argument is the same as in \cref{lem:level-extension}, and the point is that \RPSim can play the role of \Replacement. Suppose \LevelExtension fails. Then for some set $A$ of urelements, by Potent Induction there is an $\in$-minimal level $s$ with $\ker(s) \subseteq A$ that has no extension with kernel $A$. Let
$$S=\{r\in s\mid\Lev(r)\land\ker(r)=\ker(s)\}.$$
\noindent For every $r\in S$, the minimality of $s$ gives a level $l$ such that $r\subseteq l$ and $\ker(l)=A$. Apply \RPSim to this assertion together with \StratificationMinus, with $S$ and $A$ as parameters. There is a supertransitive set $t$ satisfying \StratificationMinus and containing $S$ and $A$ such that
$$t\models\forall r\in S\exists l\bigl(\Lev(l)\land r\subseteq l\land\ker(l)=A\bigr).$$
By \cref{lem:level-absoluteness-supertransitive}, every $r\in S$ is contained in an actual level $l\in t$ with kernel $A$. Let
$$\bar h=\{l\in t\mid\Lev(l)\land\ker(l)=A\land\exists r\in S(r\subseteq l)\},$$
and extend $\bar h$ downwards by setting
$$h=\{l\mid\Lev(l)\land\ker(l)=A\land(l\in\bar h\lor\exists w\in\bar h(l\in w))\}\cup A.$$
The same argument as in \cref{lem:level-extension} shows that $h$ is a history with kernel $A$ and that $s\subseteq\pot(h)$, which contradicts the assumption of $s$.\end{proof}

\begin{lemma}[\STU]\label{lem:stage-rp-minus-slices-bounded}
\StageRPMinus implies that every slice is a subset of a level.
\end{lemma}

\begin{proof}
Work in \STU + \StageRPMinus. By Theorem \ref{thm:stu-proves-ltu-minus} and Proposition \ref{prop:stage-rp-minus-rpsim}, we can use all consequences of \LTUMinus + \RPSim including \LevelExtension by Lemma \ref{lem:rpsim-level-extension}. Suppose some slice is not a subset of any level, and by Potent Induction, fix an $\in$-minimal slice $x$ as such. Let $A=\ker(x)$ and $H_x=\{q\in x\mid\Slice(q)\}$.

For each $q\in H_x$, minimality gives a level $m$ containing $q$. By the \hyperref[lem:level-shrink]{Level-Shrink Lemma}, $m^{\ker(q)}$ is a level with kernel $\ker(q)$ containing $q$. Since $\ker(q)\subseteq A$, we can apply \LevelExtension given by \cref{lem:rpsim-level-extension} to get a level $l$ with kernel $A$ containing $q$.

Thus $\forall q\in H_x\exists l(\Lev(l)\land q\subseteq l\land\ker(l)=A)$. Apply \RPSim to this assertion and \StratificationMinus. It follows from \cref{lem:level-absoluteness-supertransitive} that there is a set $\bar h$ of such levels covering every member of $H_x$. Extending $\bar h$ downwards through the levels with kernel $A$ and adding $A$ as in \cref{lem:rpsim-level-extension} gives a history $h$ with kernel $A$. Let $t=\pot(h)$. Consequently $x\subseteq t$ by \cref{lem:stu-slice-decomposition}---contradiction.\end{proof}

\begin{lemma}\label{lem:stage-rp-minus-proves-ltu-rp}
\STU+\StageRPMinus proves \LTU+\RPSim.
\end{lemma}

\begin{proof}
By \cref{prop:stage-rp-minus-rpsim}, \RPSim holds. It remains to prove \Stratification. Fix an object $x$ and a stage $\mathbf{s}$ such that $x\form\mathbf{s}$. By \cref{lem:stage-rp-minus-slices-bounded}, there is a level $t$ such that $\check{\mathbf{s}}\subseteq t$. If $x$ is a set, then $x\subseteq\check{\mathbf{s}}\subseteq t$ by \cref{lem:stu-slice-properties}(2). If $x$ is a urelement, then $x\in\check{\mathbf{s}}\subseteq t$. Thus \Stratification holds.
\end{proof}

We now prove the other direction of the equivalence. Note that the canonical interpretation which treats stages as levels would require \LTU+\RPSim to prove \LevelRP, which is impossible as \LevelRP implies \UrSet while \ZFU + \RPSim does not. So we devise a different interpretation to cope with this issue.

\begin{definition}\label{def:semi-level}
A set $s$ is a \textit{semi-level}, written $\SemiLev(s)$, iff
\begin{enumerate}
\item $s\text{ is supertransitive}$;
\item $\forall x\in s(\Set(x)\rightarrow\exists r\in s(\Lev(r)\land x\subseteq r)).$
\end{enumerate}
\end{definition}

\noindent Every level is a semi-level by \cref{lem:levels-potent-transitive,lem:level-decomposition}. A semi-level need not arise from a single history or have a single kernel. Its set-members are instead locally bounded by levels belonging to it.

\begin{lemma}\label{thm:ltu-rpsim-interprets-stu-stage-rpminus}\
\begin{enumerate}
\item \LTU+\RPSim interprets \STU+\StageRPMinus.

\item \LTU+\RP interprets \STU+\StageRP.
\end{enumerate}
\end{lemma}

\begin{proof}
We define the following translation $\dagger$.
\begin{align*}
\Stage(\mathbf{s})^\dagger&\equiv\SemiLev(\mathbf{s}),\\
(\mathbf{r}<\mathbf{s})^\dagger&\equiv\SemiLev(\mathbf{r})\land\SemiLev(\mathbf{s})\land\mathbf{r}\in\mathbf{s},\\
(x\form\mathbf{s})^\dagger&\equiv\bigl(\Set(x)\land x\subseteq\mathbf{s}\bigr)\lor\bigl(\Ur(x)\land x\in\mathbf{s}\bigr),\\
(\forall\mathbf{s}\,\varphi)^\dagger&\equiv\forall\mathbf{s}(\SemiLev(\mathbf{s})\rightarrow\varphi^\dagger).
\end{align*}
\noindent It is routine to verify that under this translation availability is exactly membership:
$$\bigl(x\avail\mathbf{s}\bigr)^\dagger\quad\longleftrightarrow\quad x\in\mathbf{s}.$$
Moreover, by Level Absoluteness and the supertransitivity of $t$, semi-levelhood is absolute for every semi-level $t$. As a result, relativization and translation commute: for all objects $\vec x$ and semi-levels $\vec r$ in a semi-level $t$, we have
$$\bigl(\varphi^\textbf{t}(\vec{x},\vec{\mathbf r})\bigr)^\dagger\quad\longleftrightarrow\quad\bigl(\varphi^\dagger(\vec{x},\vec{\mathbf r})\bigr)^t$$
by induction on stage formulas.

(1) We verify that \LTU+\RPSim proves all the translated axioms of \STU+\StageRPMinus. $\mathsf{Order}^\dagger$ holds by transitivity of semi-levels. For $\mathsf{Priority}^\dagger$, suppose that $x$ occurs at a semi-level $\mathbf{s}$ and $y\in x$. Then $x$ is a set and $x\subseteq\mathbf{s}$, so $y\in\mathbf{s}$ and hence, $y$ is available at $\mathbf{s}$. For $\mathsf{Specification}^\dagger$, suppose that every object satisfying $\theta$ is available at $\mathbf{s}$. So every such object belongs to $\mathbf{s}$. \Separation then gives the set $y=\{x\in\mathbf{s}\mid\theta^\dagger(x)\}$, which is formed on $\mathbf{s}$.

For $\mathsf{Staging}^\dagger$, fix an object $x$ and apply \RPSim to get a supertransitive set $s$ such that $x\in s$ and (\Stratification)$^s$. By \cref{lem:level-absoluteness-supertransitive}, every level recognized by $s$ is an actual level. Hence every set in $s$ is bounded by a level in $s$, so $s$ is a semi-level. Hence $\Stage(\mathbf{s})^\dagger$ and $(x\form\mathbf{s})^\dagger$.

Finally, consider \StageRPMinus. Fix semi-levels $\vec r$ and suppose that $\varphi^\dagger(\vec{x},\vec{\mathbf r})$ holds. Apply \RPSim to $\varphi^\dagger(\vec{x},\vec{\mathbf r})\land\Stratification.$
There is a supertransitive set $t$ containing $\vec{x},\vec r$ such that 
$$\bigl(\varphi^\dagger(\vec{x},\vec{\mathbf r})\bigr)^t \land \mathsf{Stratification}^t.$$
It follows that $t$ is a semi-level. Consequently, $(\exists \textbf{t}(\varphi^\textbf{t}(\vec{x},\vec{\mathbf r}) \land \vec{x}\avail \textbf{t}\land \vec {\textbf{r}} < \textbf{t}))^\dagger$, which is the translated conclusion of \StageRPMinus. 

(2) By \cref{diagram:ltu-implications}, \LTU+\RP proves \RPSim, so by (1) it remains only to verify the translated instances of \StageRP. Fix an $\mathcal{L}_\textbf{stage}$-formula $\varphi(\vec{x},\vec{\mathbf r})$ and semi-levels $\vec u$. By finite simultaneous reflection, there is a transitive set $t$ containing $\vec u$ which reflects $\varphi^\dagger$, \Stratification, \Powerset, and the formula defining the powerset of a set. The same argument for \RP implies \RPSim as in \cref{prop:LTUImplication} shows that $t$ is supertransitive, and the argument in (1) shows that $t$ is a semi-level.

Now let $\vec x\in t$ be objects and let $\vec r\in t$ be semi-levels. Since $t$ reflects $\varphi^\dagger$,
$$\varphi^\dagger(\vec{x},\vec{\mathbf r})\quad\longleftrightarrow\quad\bigl(\varphi^\dagger(\vec{x},\vec{\mathbf r})\bigr)^t.$$
Since $\dagger$ and relativization commute, the right-hand side is equivalent to $\bigl(\varphi^\textbf{t}(\vec{x},\vec{\mathbf r})\bigr)^\dagger$. Hence $\varphi^\dagger(\vec{x},\vec{\mathbf r})\leftrightarrow\bigl(\varphi^\textbf{t}(\vec{x},\vec{\mathbf r})\bigr)^\dagger.$ It follows that the $\dagger$ translation of the following holds
$$\vec{\textbf{u}} < \textbf{t} \land \forall \vec x \avail \textbf{t} \forall \vec{\textbf{r}}< \textbf{t}(\varphi(\vec{x},\vec{\mathbf r})\leftrightarrow \varphi^\textbf{t}(\vec{x},\vec{\mathbf r}))$$
which yields the translated instance of \StageRP.
\end{proof}

\begin{theorem}\label{thm:partial-reflection-set-theoretic-equivalence}
The following pairs of theories are set-theoretically equivalent.
\begin{enumerate}
\item \STU+\StageRPMinus and \LTU+\RPSim.

\item \STU+\StageRP and \LTU+\RP.
\end{enumerate}
\end{theorem}

\begin{proof}
By \cref{lem:stage-rp-minus-proves-ltu-rp,thm:ltu-rpsim-interprets-stu-stage-rpminus}, the theories in (1) interpret each other by translations that leave $\mathcal{L}_\Ur$ unchanged. Hence they are set-theoretically equivalent. It remains to verify that \STU+\StageRP proves \LTU+\RP. 

Since \StageRP implies \StageRPMinus, \cref{lem:stage-rp-minus-proves-ltu-rp} gives \LTU. Let $\varphi(\vec{x})$ be an $\mathcal{L}_\Ur$-formula. By \StageRP, there is a stage $\mathbf{s}$ such that $\forall\vec{x}\avail\mathbf{s}\bigl(\varphi(\vec{x})\leftrightarrow\varphi^{\mathbf{s}}(\vec{x})\bigr).$ Let $t=\check{\mathbf{s}}$. $t$ is supertransitive by \cref{lem:stu-slice-properties}, and object quantification in $\varphi^{\mathbf{s}}$ is exactly quantification over $t$. Hence $\forall\vec{x}\in t\bigl(\varphi(\vec{x})\leftrightarrow\varphi^t(\vec{x})\bigr).$ Thus \RP holds.\end{proof}

Now as an application of the set-theoretic equivalence results, we prove that the stage axioms form a strict hierarchy over \STU.

\begin{theorem}\label{thm:stu-implication-diagram-complete}
\Cref{diagram:stu-implications} is complete over \STU.
\end{theorem}

\begin{proof}
By \cref{thm:stage-directed-not-ltu}, \StageDirected does not imply \Stratification. But \Stratification follows from both \StageUnbounded and \StageRPMinus by \cref{thm:stage-level-unbounded-equivalence} and \cref{lem:stage-rp-minus-proves-ltu-rp}, respectively. Thus \StageDirected proves no other axiom in the diagram. \StageUnbounded does not imply \StageDirected either. This is because the one-object model in Theorem \ref{thm:ltu-implication-diagram-complete} shows that  \LTU + \LevelUnbounded does not prove \Pairing, so neither does \STU + \StageUnbounded by Theorem \ref{thm:stage-level-unbounded-equivalence}; but \StageDirected implies \Pairing by Proposition \ref{prop:stage-axiom-consequences}.

Now pair each remaining axiom $\varphi$ with its level counterpart $\varphi'$. Suppose \STU proves an additional implication from $\varphi$ to $\psi$ in \cref{diagram:stu-implications}. By \cref{thm:partial-reflection-set-theoretic-equivalence}, or  \cref{thm:stage-level-unbounded-equivalence}, we have \STU + $\psi$ proves $\psi'$ and so \STU + $\varphi$ proves $\psi'$. Moreover, every instance of $\psi'$ is an $\mathcal{L}_\Ur$-formula, so by the relevant set-theoretic equivalence theorem again, it follows that \LTU + $\varphi'$ proves $\psi'$. But this contradicts the completeness of \cref{diagram:ltu-implications}, proved in \cref{thm:ltu-implication-diagram-complete}. Therefore, there is no further implication.
\end{proof}
\noindent In fact, the proof shows that adding any two axioms in \Cref{diagram:stu-implications} to \STU respectively yields theories that are not set-theoretically equivalent.

\section{Quasi-Categoricity}\label{section:quasi-categoricity}
Zermelo \cite{Zermelo1930GrenzzahlenMengenbereiche} proved the quasi-categoricity of second-order \ZF: any full model of second-order \ZF is isomorphic to some $V_\kappa$, where $\kappa$ is an inaccessible cardinal. Button and Walsh \cite[Section 8 C]{ButtonWalsh2018PhilosophyModelTheory} proved that second-order \LT is itself quasi-categorical: any full model of second-order \LT is isomorphic to some hierarchy $V_\alpha$. Categoricity is often regarded as a criterion for assessing the robustness of a theory, or of the conception underlying it, since a categorical theory \textit{pins down} the objects it describes up to isomorphism. The result of Button and Walsh, together with the set-theoretic equivalence between \LT and \ST, therefore suggests that the Basic Iterative Story provides a robust conception of set when restricted to pure sets. We have seen that the failure of set-theoretic equivalence in the presence of urelements and considered how extensions of the Basic Iterative Story can remedy it. In this section, we study the models of \LTU and investigate, through the lens of categoricity, the robustness of the Basic Iterative Story and its extensions.

\LTU2 is \LTU formulated in the second-order language in which \Separation is replaced by the second-order axiom \SeparationTwo. A \textit{full model} $M$ of a second-order urelement set theory interprets the second-order variables as ranging over all subsets of its first-order domain. Since \LTU proves \Foundation (by Theorems \ref{thm:stu-foundation} and \ref{thm:ltu-interprets-stu}), the standard argument shows that every full model of \LTU2 is externally well-founded. For simplicity, our background theory will be \ZFU plus an additional assumption that every well-founded model is isomorphic to a transitive one, so it suffices to consider transitive models. But all the arguments for transitive models can be transferred to well-founded models, so the additional assumption is not essential.

For any well-founded model $M$, let the \textit{height} of $M$ be the order type of the ordinals of $M$, and let the \textit{width} of $M$ be the cardinality of $\Ur^M$. A natural notion of quasi-categoricity in the presence of urelements is the following.
\begin{definition}
Let $T$ be a second-order theory in $\mathcal{L}_\Ur$. $T$ is \textit{strongly quasi-categorical} if any two full models of $T$ with the same height and width are isomorphic.
\end{definition}
\noindent The idea is that full models of a strongly quasi-categorical theory should be completely determined by the number of urelements and how far the iterative process has gone. There is, however, a weaker notion of quasi-categoricity.
\begin{definition}
Two models $N$ and $M$ are said to have the same \textit{level-width} if there is a bijection $\pi: \Ur^M \to \Ur^N$ such that for every $A\subseteq \Ur^M$, $$A\in M \text{ iff } \pi[A]\in N.$$
\end{definition}
\noindent That is, the bijection $\pi$ between their urelements induces a bijection between their sets of urelements.
\begin{definition}
Let $T$ be a theory in $\mathcal{L}_\Ur$. $T$ is \textit{weakly quasi-categorical} if any two full models of $T$ with the same height and level-width are isomorphic.
\end{definition}
\noindent In other words, full models of a weakly quasi-categorical theory are determined by the initial hierarchies and how far the iterative process has gone.

Next we show that \LTU2 is not weakly quasi-categorical because it admits uneven models, while \LevelExtension restores weak quasi-categoricity. 
\begin{theorem}\label{thm:unions-of-hierarchies}
Let $\mathcal{S}$ be a non-empty set of non-initial hierarchies. Then $\bigcup \mathcal{S}$ is a full model of \LTU2.
\end{theorem}

\begin{proof}
Let $M=\bigcup\mathcal{S}$, with membership and urelementhood inherited from the ambient universe. Thus \UrDef and \Extensionality hold in $M$. \SeparationTwo holds in $M$: if $y \subseteq x \in V_\alpha(A) \in M$ for some hierarchy $V_\alpha(A)\in\mathcal{S}$, the potency of $V_\alpha(A)$ gives $y\in V_\alpha(A)\subseteq M$.
$M$ is a union of hierarchies and hence supertransitive. Moreover, $M \models$ \StratificationMinus, since every set in $M$ is a subset of some supertransitive set, and being supertransitive is absolute for supertransitive sets. Thus levelhood is absolute for $M$ by \cref{lem:level-absoluteness-supertransitive}. To verify that $M \models$ \Stratification, let $x \in M$. If $x$ is a set, then $x \subseteq V_\beta (A) \in V_\alpha (A)$ for some non-initial hierarchy $V_\alpha (A) \in M$; $V_\beta (A)$ is then a level in $M$ by Theorem \ref{thm:levels-exactly-hierarchies} which contains $x$.  If $x$ is some urelement $a$ in $M$, $a \in A \in V_\alpha(A) \in \mathcal{S}$ for some non-initial hierarchy $V_\alpha(A)$. So $M$ thinks that $a$ is in the level $A$.\end{proof}

\begin{corollary}
\LTU2 is not weakly quasi-categorical.
\end{corollary}
\begin{proof}
Now consider three urelements $a,b,c$  and put $A=\{a\}$ and $B=\{b,c\}$. Let
$$M=V_\omega(A)\cup V_1(B),\qquad N=V_\omega(A)\cup V_2(B).$$
By \cref{thm:unions-of-hierarchies}, they are both full  models of \LTU2 with height $\omega$. They have the same sets of urelements and hence the same level-width. $B$ is definable in both models as the unique two-element set of urelements. But $N$ thinks that $B$ has a powerset, while $M$ does not. So they are not isomorphic.\end{proof}

\begin{lemma}\label{lem:level-extension-reconstruction}
Let $M$ be a transitive full model of \LTU2 + \LevelExtension. Then
$$M=\bigcup\{V_\alpha(A)\mid A\in M\land A\subseteq \Ur^M\},$$
where $\alpha$ is the height of $M$.
\end{lemma}

\begin{proof}
First note that \SeparationTwo makes $M$ supertransitive and so levelhood is absolute for $M$ by \cref{lem:level-absoluteness-supertransitive}. For the inclusion $M\subseteq\bigcup\{V_\alpha(A)\mid A\in M\land A\subseteq\Ur^M\}$, let $x\in M$ be a set. By \Stratification in $M$, there is a level $s \in M$ with $x \subseteq s$. Let $A = \ker(s)$. Since $A \subseteq s\in M$ and $M$ is supertransitive, $A \in M$; moreover, $A \subseteq \Ur^M$. By \cref{thm:levels-exactly-hierarchies}, $s = V_\beta (A)$ for some ordinal $\beta$. Since $\beta \subseteq V_\beta (A) \in M$ and $M$ is supertransitive, $\beta \in M$ and hence $\beta < \alpha$. Thus, $x \subseteq V_\beta(A) \in V_\alpha(A)$ and so $x \in V_\alpha(A)$. The urelement case is immediate by \Stratification.

Conversely, let $A\in M$ be a set of urelements and let $x\in V_\alpha(A)$ be a set. Fix $\beta<\alpha$ such that $x\subseteq V_\beta(A)$. $\beta \in M$ so by \Stratification, there is a level $r\in M$ such that $\beta\subseteq r$. The empty-kernel restriction $r^\emptyset$ is a level by the \hyperref[lem:level-shrink]{Level-Shrink Lemma}, and $\beta\subseteq r^\emptyset$. By \LevelExtension in $M$, there is a level $t\in M$ such that $r^\emptyset\subseteq t$ and $\ker(t)=A.$ By \cref{thm:levels-exactly-hierarchies}, $t=V_\gamma(A)$ for some ordinal $\gamma$. Since $\beta\subseteq r^\emptyset\subseteq t$, we have $\beta\leq\gamma$, so $V_\beta(A)\subseteq t$. Therefore $x \subseteq V_\beta(A) \subseteq t \in M$ and hence $x\in M$.\end{proof}

\begin{theorem}\label{thm:level-extension-weak-quasi-categoricity}
\LTU2 + \LevelExtension is weakly quasi-categorical.
\end{theorem}
\begin{proof}
Let $M$ and $N$ be two transitive full models of \LTU2 + \LevelExtension with the common height $\alpha$, and let $\pi:\Ur^M\longrightarrow \Ur^N$ witness that they have the same level-width. By \cref{lem:level-extension-reconstruction} we have
\begin{align*}
M&=\bigcup\{V_\alpha(A)\mid A\in M\land A\subseteq\Ur^M\},\\
N&=\bigcup\{V_\alpha(B)\mid B\in N\land B\subseteq\Ur^N\}.
\end{align*}
Let $V(\Ur^M)$ denote the cumulative universe generated by $\Ur^M$. Using well-founded recursion, extend $\pi$ (externally) by letting $\pi(x)=\{\pi(y)\mid y\in x\} = \pi [x]$ for every set $x$ in $V(\Ur^M)$. $\pi$ preserves membership and urelementhood and is injective by an easy induction. It remains to show that $\pi$ maps $M$ onto $N$.

Note that for every $A\subseteq\Ur^M$, $\pi$ sends $V_\alpha(A)$ bijectively onto $V_\alpha (\pi [A])$. Now let $x\in M$ be a set. There is some $A\in M$ with $A\subseteq\Ur^M$ such that $x\in V_\alpha(A)$. So $\pi(x) \in V_\alpha(\pi[A])$. Since $\pi[A] \in N$, $\pi(x) \in N$. Thus $\pi$ maps $M$ into $N$. Conversely, let $y\in N$. There is some $B\in N$ with $B\subseteq\Ur^N$ such that $y\in V_\alpha(B)$. Put $A=\pi^{-1}[B]$. Then $A\in M$, and $\pi$ maps $V_\alpha(A)$ bijectively onto $V_\alpha(B)$. Hence $y=\pi(x)$ for some $x\in V_\alpha(A)\subseteq M$. Thus $\pi$ is onto $N$. This completes the proof.\end{proof}

\begin{corollary}\label{cor:extensions-weakly-quasi-categorical}
\LTU2 + $\varphi$ is weakly quasi-categorical, where $\varphi$ is any principle in \cref{diagram:ltu-implications}.
\end{corollary}
\begin{proof}
In view of \cref{thm:level-extension-weak-quasi-categoricity}, it suffices to show that every axiom in \cref{diagram:ltu-implications} implies \LevelExtension. By \cref{lem:level-extension}, \Replacement implies \LevelExtension. It remains to show that \LevelDirected implies \LevelExtension over \LTU.

Fix a level $s$ and a set $A$ of urelements such that $\ker(s)\subseteq A$. By \LevelDirected, there is a level $m$ such that $s,A\in m$. So $A\subseteq\ker(m)$ and the \hyperref[lem:level-shrink]{Level-Shrink Lemma} gives a level $m^A$ with kernel $A$ such that
$$m^A = \{x \in m \mid \Set(x) \land \ker(x) \subseteq A\} \cup A.$$
Thus, $s \subseteq m^A$, as required.\end{proof}

We now turn to the stronger notion of quasi-categoricity. The key axiom here is second-order \Replacement.
\begin{itemize}
\item []  (\Replacement2) \( \forall F, w (\text{fun} (F) \land \Set(w)  \to F[w] \text{ is a set})\).
\end{itemize}
\noindent Here $\text{fun} (F)$ abbreviates ``$F$ is a class function defined on all objects''. Over \LTU2, \Replacement2 is equivalent to second-order \LevelUnbounded, which says that for any class function $F$ and set $w$, $F[w]$ is contained in some level.

\begin{theorem}\label{thm:unbounded-strong-quasi-categoricity}
Assume \AC. \LTU2 + \Replacement2\ is strongly quasi-categorical.
\end{theorem}

\begin{proof}
Let $M$ and $N$ be transitive full models of \LTU2 + \Replacement2 with common height $\alpha$, and fix a bijection $\pi:\Ur^M\longrightarrow \Ur^N$. We claim that for every $A\subseteq \Ur^M$,
$$A\in M\quad\longleftrightarrow\quad\pi[A]\in N.$$
Suppose $A\in M$. There is no surjection from $A$ onto $\alpha$, since otherwise \Replacement2 in $M$ would place $\alpha$ in $M$. By \AC in the background theory, $A$ is therefore equinumerous with some $\beta<\alpha$, and so $\beta \in N$. The corresponding bijection from $\beta$ onto $\pi[A]$ then places $\pi[A]$ in $N$ by  \Replacement2 in $N$. The converse follows by applying the same argument to $\pi^{-1}$. 

By \cref{lem:level-extension}, both models satisfy \LevelExtension. Thus $M$ and $N$ are two full models of \LTU2 + \LevelExtension of the same height and the same level-width. So they are isomorphic by \cref{thm:level-extension-weak-quasi-categoricity}.\end{proof}
\noindent The use of \AC in the background theory is essential: assuming the consistency of an inaccessible cardinal, second-order \ZFU (\ZFU2) is consistently not strongly quasi-categorical. In fact, in the choiceless context, two full models of \ZFU2 with the same height and the same urelements can disagree on whether the urelements form a set (see \cref{thm:zfu-not-strongly-quasi-categorical} in Appendix). This is unsurprising: no axiom of \ZFU2 is able to determine which subclasses of urelements become the initial levels. So having the same number of urelements from the \textit{external perspective} cannot determine how two models generate their sets. That said, the external use of \AC simply reduces the situation to weak quasi-categoricity, as the proof of Theorem \ref{thm:unbounded-strong-quasi-categoricity} shows. Thus strong quasi-categoricity for urelement set theory is not a metatheoretically neutral notion.

We conclude by considering how first-order reflection principles may yield strong quasi-categoricity. One must decide whether class parameters are allowed. Write \RPSimCl and \RPCl for the corresponding schemes that allow class parameters, where a class parameter $X$ is reinterpreted in the reflecting set $t$ as $X\cap t$. Class parameters make these principles substantially stronger. Suppose that $R$ is a class relation and $w$ is a set such that $\forall x\in w\,\exists y\,R(x,y)$. Applying \RPSimCl to this first-order formula gives a supertransitive set $t$ containing $w$ such that $\forall x\in w\,\exists y\in t\,R(x,y).$ Thus $t$ is a collection set. Together with the class-parameter versions of the implications proved in \cref{prop:LTUImplication}, this gives
$$\RPCl\ \longrightarrow\ \RPSimCl\ \longrightarrow\ \CollectionTwo\ \longrightarrow\ \Replacement2.$$
Hence, assuming \AC in the background theory, \LTU2 with any of these reflection principles is strongly quasi-categorical by \cref{thm:unbounded-strong-quasi-categoricity}. Without class parameters the situation is different: \LTU2 + \RP is not strongly quasi-categorical even assuming \AC (see \cref{thm:zfu-not-strongly-quasi-categorical} in Appendix). Consequently, while \LevelExtension suffices for weak quasi-categoricity, it does not yield strong quasi-categoricity. In other words, \Replacement2, motivated by the unbounded principle, appears essential for strong quasi-categoricity in the presence of urelements.

\section{Concluding Remarks}
We investigated the iterative conception of set in its most basic and general form. We argued that \UrStage, the assumption that all urelements are available at one stage, has often been mistakenly regarded as part of the basic iterative conception. The iterative conception, as articulated in the Basic Iterative Story, is neutral about the number of urelements and whether they form a set.

This neutrality has mathematical and philosophical consequences. The Basic Iterative Story can be formalized in terms of stages or levels, yielding the theories \STU and \LTU. In the presence of urelements, the two formalizations are no longer set-theoretically equivalent. \STU cannot prove \LTU's key axiom, \Stratification, providing the first indication that the Basic Iterative Story fails to provide a robust conception of set in its most general form.

We then considered different extensions of the Basic Iterative Story inspired by the ideas of directedness, unboundedness, and reflection. We isolated the core principle \LevelExtension and showed that all the stronger principles imply it. As a result, the unboundedness and reflection principles restore set-theoretic equivalence between the corresponding stage and level theories. Moreover, using the relevant independence results in \ZFU together with these set-theoretic equivalences, we showed that these further principles form two strict hierarchies.

The discussion of quasi-categoricity provides further evidence that the Basic Iterative Story is not robust. We distinguished two notions of quasi-categoricity. \LTU2 is not even weakly quasi-categorical, while adding \LevelExtension restores weak quasi-categoricity. The stronger notion of quasi-categoricity raises metatheoretic issues. Assuming \AC in the background theory, \Replacement2 yields strong quasi-categoricity. The background assumption of \AC cannot simply be omitted: in the choiceless context, two full models of \ZFU2 with the same height and the same urelements can disagree on whether the urelements form a set.

\appendix
\renewcommand{\sectionname}{}
\renewcommand{\thesection}{Appendix}
\section{Countermodels to Strong Quasi-Categoricity}

\begin{theorem}\label{thm:zfu-not-strongly-quasi-categorical}
Assume the consistency of an inaccessible cardinal.
\begin{enumerate}
\item It is consistent with \ZFU that \ZFU2 is not strongly quasi-categorical.

\item It is consistent with \ZFCU that \LTU2 + \RP is not strongly quasi-categorical.
\end{enumerate}
\end{theorem}

\begin{proof}
(1) Let $U$ be a model of \ZFCU with an inaccessible cardinal $\kappa$. Assume that $\Ur^U$ is a countably infinite set $A$ in $U$. We first construct the basic Fraenkel model in $U$. Let $G$ be the group of all permutations of $A$. An object $x$ is \textit{symmetric} if there is a finite $E\subseteq A$ such that $\pi x = x$ for every $\pi \in G$ that pointwise fixes $E$, where $\pi$ is understood as its canonical extension as in Theorem \ref{thm:level-extension-weak-quasi-categoricity}; such an $E$ is a \textit{support} of $x$. A set $x$ is \textit{hereditarily symmetric} if $x$ and every object in its transitive closure are symmetric. Let $W$ be the class of hereditarily symmetric objects. By standard arguments, $W$ is a transitive inner model of \ZFU which contains the same pure sets and urelements as $U$, and $A$ is a non-well-orderable set in $W$.

We now work in $W$, and define
$$N=V_\kappa(A),\qquad M=\bigcup \{V_\kappa(B) \mid B \subseteq A \text{ is finite}\}.$$
We show that $M$ and $N$ are full models of \ZFU2 in the sense of $W$: their second-order variables range over all sets of their domains. We only show that \Replacement2 holds in both models, as it is routine to verify the other axioms by the usual hierarchy argument.

First consider $N$. Let $w\in N$ be a set and let $F:w \to N$ be any function in $W$. Then $F[w] \subseteq N$. In $U$, we have $|F[w]| \leq |w| < \kappa$ so it follows from the inaccessibility of $\kappa$ that the rank of $F[w]$ is less than $\kappa$. Thus, $F[w] \in N$. Therefore $N$ satisfies \Replacement2 and hence is a full model of \ZFU2.

To verify \Replacement2 in $M$, let $F:w \to M$ be any function in $W$. Fix finite $B\subseteq A$ with $w \in V_\kappa(B)$, and fix a finite support $E$ for $F$. Put $D=B\cup E$. We claim that $\ker(F[w]) \subseteq D$. Otherwise, there is some $x \in w$ such that $\ker(\{F(x)\}) \not\subseteq D$. Fix some $a \in \ker(\{F(x)\}) - D$. Since $F(x) \in M$, $\ker(\{F(x)\})$ is finite and so we can fix some $b \in A - (\ker(\{F(x)\}) \cup D)$. Let $\pi$ swap $a$ and $b$. $\pi$ is an automorphism that pointwise fixes $D$, so it fixes $F$ and every member of $w$. Thus $\pi F(x) = \pi F (\pi x) = F(x)$, but $\pi F(x) \neq F(x)$ as $\pi$ moves $a$ out of the kernel $\ker(\{F(x)\})$---contradiction.

Moreover, the rank of $F[w]$ is less than $\kappa$ by the previous argument, so $W \models F[w] \in V_\kappa(D)$ and so $F[w] \in M$. Thus $M$ satisfies \Replacement2 and is a full model of \ZFU2.

Finally, both $M$ and $N$ have height $\kappa$, and $\Ur^M=\Ur^N=A$. But $N\models\UrSet$ and $M\models\neg\UrSet$, so $M$ and $N$ are not isomorphic. Therefore, \ZFU2 is not strongly quasi-categorical in $W$.

(2) Let $U$ be a model of \ZFCU with an inaccessible cardinal $\kappa$ and a set $A$ of urelements of size $\aleph_1$. Work in $U$ and put
$$N=V_\kappa(A),\qquad M=\bigcup\{V_\kappa(B)\mid B\subseteq A\text{ is countable}\}.$$
$N$ satisfies \ZFCU2 so it has both \LTU2 and \ZFCU. It then follows from Yao \cite[Theorem 2.16]{Yao2026AxiomatizationForcing} that $M$ also satisfies first-order \ZFCU. And it follows from \cref{thm:unions-of-hierarchies} that $M$ satisfies \LTU2. Both models satisfy \RP which only allows first-order parameters. This is because both models satisfy the axiom \Tail introduced in \cite{Yao2026AxiomatizationForcing}, which implies \RP over \ZFCU by \cite[Corollary 2.12.1]{Yao2026AxiomatizationForcing}.

$M$ and $N$ have the same height and the same urelements. But $N\models\UrSet$, whereas $M\models\neg\UrSet$. Thus $M$ and $N$ are not isomorphic, so \LTU2+\RP is not strongly quasi-categorical in $U$.\end{proof}

\printbibliography

\end{document}